\documentclass[a4paper,11pt,reqno]{amsart}

\usepackage[T1]{fontenc}
\usepackage[utf8]{inputenc}

\usepackage{amsmath,amsfonts,amsthm,amssymb}
\usepackage[showonlyrefs]{mathtools}
\usepackage{mathrsfs}
\DeclarePairedDelimiter\norm{\lVert}{\rVert}

\usepackage[english]{babel}

\makeatletter
\@namedef{subjclassname@2020}{%
  \textup{2020} Mathematics Subject Classification}
\makeatother

\allowdisplaybreaks[2]

\theoremstyle{thmstyleone}%
\newtheorem{theorem}{Theorem}[section]

\newtheorem{lemma}[theorem]{Lemma}

\theoremstyle{thmstylethree}%
\theoremstyle{definition}

\theoremstyle{thmstyltwo}%

\numberwithin{equation}{section}

\begin{document}

\title[Jordan isomorphisms of group algebras]{Isometric Jordan isomorphisms of group algebras}

\author{J. Alaminos}
\author{J. Extremera}
\author{C. Godoy}
\author{A.\,R. Villena}

\address{J. Alaminos, J. Extremera and A.\,R. Villena, Departamento de An\' alisis
Matem\' atico, Fa\-cul\-tad de Ciencias, Universidad de Granada,
 Granada, Spain} 
\email{alaminos@ugr.es, 
 jlizana@ugr.es,
 avillena@ugr.es}

\address{C. Godoy, Departamento de Matem\'aticas, Universidad de Alicante, Crta. San Vicente del Raspeig s/n, San Vicente del Raspeig, 
	Alicante, Spain} 
\email{c.godoy@ua.es}

\begin{abstract}
	Let $G$ and $H$ be locally compact groups.
	We  will show that each contractive Jordan isomorphism $\Phi\colon L^1(G)\to L^1(H)$ is either an isometric isomorphism or an isometric anti-isomorphism. 
	We will apply this result to study isometric two-sided zero product preservers on group algebras and, further, to study local and approximately local isometric automorphisms of group algebras.
\end{abstract}

\thanks{The authors were supported by 
MCIN/AEI/10.13039/501100011033 and ``ERDF A way of making Europe'' grant PID2021-122126NB-C31
and by Junta de Andaluc\'{\i}a grant FQM185.}



\maketitle

\section{Introduction}

Jordan homomorphisms appear in a wide variety of seemingly dis\-pa\-rate settings.
Numerous linear preserver problems lead to Jordan homomorphisms:
invertibility preservers (\cite{Au,Au2}),
two-sided zero product preservers (\cite{ABEV, AEV, BGV}),
commutativity preservers, normality preservers (\cite[Chapter~7]{BCM}),
preservers on quantum structures (\cite{Mo}),
to mention a few of them.
The surjective isometries between $C^*$-algebras are associated to Jordan $\ast$-isomorphisms~\cite{Kadison}
and the surjective isometries between noncommutative $L^p$ spaces correspond to Jordan $\ast$-isomorphisms between the
underlying von Neumann algebras~\cite{Sh}.
Local homomorphisms also lead to Jordan homomorphisms~\cite{ABEGV}.

Let $\mathcal{A}$ and $\mathcal{B}$ be complex algebras.
A linear map $\Phi\colon\mathcal{A}\to\mathcal{B}$ is called a \emph{Jordan homomorphism} if
\[
	\Phi(a^2)=\Phi(a)^2\quad\forall a\in\mathcal{A};
\]
equivalently,
\[
	\Phi(a\circ b)=\Phi(a)\circ\Phi(b)\quad\forall a,b\in\mathcal{A},
\]
where, from now on, $\circ$ stands for the so-called Jordan product. The Jordan product is defined on
any complex algebra $\mathcal{A}$ by
\[
	a\circ b=\tfrac{1}{2}(ab+ba)\quad\forall a,b\in\mathcal{A}.
\]
The meaning of concepts like Jordan isomorphism and Jordan automorphism
(or Jordan $\ast$-isomorphism and Jordan $\ast$-automorphism,
in the case where $\mathcal{A}$ and $\mathcal{B}$ are equipped with an involution $\ast$)
are supposed to be clear.
Homomorphisms and anti-homomorphisms are obvious examples of Jordan homomorphisms,
and the basic problem is whether every Jordan homomorphism can be expressed through these standard examples.
A breakthrough in this problem was obtained by Herstein in~\cite{H} by showing that
every Jordan homomorphism from an arbitrary ring onto a ($2$, $3$-torsion free) prime ring is either a homomorphism or an anti-homomorphism.
It is probably worth mentioning that group algebras can be far from being prime.
For example, if $G$ is a locally compact group such that $\dim\mathcal{Z}(L^1(G))>1$, then $L^1(G)$ is not prime.
Indeed, by~\cite{LM}, 
$\mathcal{Z}(L^1(G))$ is a Tauberian, regular, semisimple, commutative Banach algebra and hence there exist
non-zero $f,g\in\mathcal{Z}(L^1(G))$ such that $f\ast g=0$, which gives $f\ast L^1(G)\ast g=\{0\}$ and so $L^1(G)$ is not prime.
Our benchmark when considering Jordan homomorphisms in the context of group algebras has been the celebrated theorem by Kadison~\cite[Theorem~10]{Kadison} stating that each Jordan $\ast$-isomorphism from a von Neumann algebra
onto a $C^*$-algebra is the direct sum of a $\ast$-isomorphism and a $\ast$-anti-isomorphism.
It should be pointed out that, on account of~\cite[Theorems~5 and~7]{Kadison},
the Jordan $\ast$-isomorphisms occurring in Kadison's theorem are exactly the isometric Jordan isomorphisms,
while the $\ast$-isomorphisms and $\ast$-anti-isomorphisms are nothing but the isometric isomorphisms and the isometric anti-isomorphisms, respectively. We can thus rephrase Kadison's theorem as:
each  isometric Jordan isomorphism from a von Neumann algebra
onto a $C^*$-algebra is the direct sum of an isometric isomorphism and an isometric anti-isomorphism.

In Section 2 we put the isometric Jordan isomorphisms between group algebras in the center of our attention
and discuss the problem of describing their form.
We are heavily motivated by Kadison's representation of the isometric Jordan isomorphisms of operator algebras
and by Wendel's representation of the isometric isomorphisms of group algebras~\cite{W1, W2}.
Naturally, our seminal aim was to obtain a representation for them
similar to the one given by Kadison in the context of operator algebras.
However, we got more than expected.
We will show that, if $G$ and $H$ are locally compact groups, then
each contractive Jordan isomorphism $\Phi\colon L^1(G)\to L^1(H)$ is
either an isometric isomorphism or an isometric anti-isomorphism.
Surprisingly there are no combinations of isomorphisms and anti-isomorphisms at all in the context of group algebras (at an isometric level).
The secret hidden behind this fortunate phenomenon lies in two crucial facts.
The first one is that group algebras faithfully reflect the personality of the underlying groups. 
The second fact now occurs at the level of group theory;
on account of~\cite{Sc}, the only maps that ``half-preserve'' the product of the groups
are exactly the homomorphisms and the anti-homomorphisms.
Next sections are devoted to apply our result about isometric Jordan isomorphisms to
study various classes of transformations on group algebras.

In~\cite{ABEV, AEV, BGV} the authors undertook the problem of characterizing Jordan homomorphisms through zero products.
Motivated by those papers,
Section~3 deals with the problem of determining the form of any surjective isometric map $\Phi\colon L^1(G)\to L^1(H)$
with the property of preserving two-sided zero products, i.e., for all $f,g\in L^1(G)$
\[
	f\ast g=g\ast f=0 \ \implies \ \Phi f\ast\Phi g=\Phi g\ast\Phi f=0.
\]
We will show that such a map is expressible as
$\Phi f=\alpha \delta_x\ast\Psi f$ for each $f\in L^1(G)$, where
$\alpha\in\mathbb{C}$ with $\vert\alpha\vert=1$,
$x$ is an element in the centre of $H$, and
$\Psi\colon L^1(G)\to L^1(H)$  is either an isometric isomorphism or an isometric anti-isomorphism. Further,
$G$ and $H$ are isomorphic as topological groups.

Section 4 is devoted to the study of local and approximately local isometric automorphisms of group algebras.
In~\cite{MoZ}, Moln\'ar and Zalar show that, with some exceptions, every local isometric automorphism of the group algebra $L^p(G)$, with $1\le p\le\infty$,
of a metrizable compact group $G$ is an isometric automorphism.
We focus attention on the case $p=1$, and 
Section~4 is motivated by the desire to study to what extent
the foregoing can be generalized to other varieties of groups.
We will show that, at the cost of requiring the surjectivity of the map,
the metrizability can be dropped and the compactness can be heavily weakened.
Specifically,
if $G$ is a unimodular locally compact group, then each surjective local isometric automorphism of $L^1(G)$
(i.e., the map agrees with some isometric automorphism at each function in the algebra) is an isometric automorphism.
Further, 
if $G$ is a maximally almost periodic group or a discrete group, then each surjective approximately local isometric automorphism
of $L^1(G)$ (the definition should be self-explanatory)
is an isometric automorphism.

\section{Isometric Jordan isomorphisms}

Let $G$ be a locally compact group.
Throughout this paper, it is always assumed that a left Haar measure $\lambda_G$ on $G$ has been chosen.
As is customary,
we write $L^1(G)$ for the group algebra of $G$, i.e.,
the usual $L^1$-Banach space with respect to $\lambda_G$ equipped with the convolution product, and
we write  $M(G)$ for the measure algebra on $G$, i.e.,
the space of complex regular Borel measures on $G$ equipped with the total variation norm and the convolution product.
Of course, $L^1(G)$ can be thought of as the two-sided closed ideal of $M(G)$ consisting of measures that are absolutely continuous with respect to $\lambda_G$.
The space $M(G)$ is identified with the dual space of $C_0(G)$ with the duality
specified by setting
\[
	\langle \mu,\phi\rangle  =  \int_H\phi(t) \, d\mu(t) \quad\forall\phi\in C_0(G), \ \forall \mu\in M(G).
\]
By~\cite[Theorem 3.3.15]{D}, the Banach space $C_0(G)$ is a Banach
$M(G)$-bimodule (we denote its module product by $\cdot$) and $M(G)$ with respect
to convolution product is the dual of $C_0(G)$ as a Banach $M(G)$-bimodule.
The \emph{strict topology} on $M(G)$ is the topology defined by the family seminorms $(p_f)_{f\in L^1(G)}$ given by
\[
	p_f(\mu) = \Vert f\ast\mu\Vert_1+\Vert \mu\ast f\Vert_1
	\quad\forall \mu\in M(G), \ \forall f\in L^1(G).
\]
Clearly, the convolution on $M(G)$ is separately continuous in the strict topology. Furthermore, $L^1(G)$  is dense in $M(G)$ in the strict topology.

For an element $t\in G$ we denote by $\delta_t$ the unit point mass measure at $t$. It should be pointed out that
\[
	\delta_s\ast\delta_t=\delta_{st}, \quad
	\delta_s\circ\delta_t=\tfrac{1}{2}\left(\delta_{st}+\delta_{ts}\right)
	\quad\forall s,t\in G.
\]
By~\cite[Proposition~3.3.41]{D}, 
the linear span of the set $\left\{\delta_t : t\in G\right\}$ is dense in $M(G)$ with respect to the strict topology.

\begin{lemma}
	Let $G$ be a locally compact group.
	Then the strict topology on $M(G)$ agrees with the topology generated
	by the family of seminorms $(q_f)_{f\in L^1(G)}$ given by
	\[
		q_f(\mu)=\left\Vert f\circ\mu\right\Vert_1
		\quad\forall \mu\in M(G), \ \forall f\in L^1(G).
	\]
\end{lemma}

\begin{proof}
	It is clear that, for each $f\in L^1(G)$,
	\begin{equation*}
		q_f\le \tfrac{1}{2}p_f.
	\end{equation*}

	We now proceed to show that for each $f\in L^1(G)$ there exist $f_1,f_2,f_3,f_4\in L^1(G)$ such that
	\begin{equation}\label{eq1653}
		p_f\le \sum_{k=1}^4\bigl(2q_{f_{k^{2}}}+4\left\Vert f_k\right\Vert_1 q_{f_k}\bigr).
	\end{equation}
	By~\cite[Theorem~II.16]{AL}, there exist $g,h\in L^1(G)$ such that $f=g\ast h\ast g$, and taking
	\begin{align*}
		f_1  =\frac{1}{\sqrt{2}}\left(g\circ h+g\right), \ 
		f_2  =\frac{i}{\sqrt{2}}\left(g\circ h-g\right), \ 
		f_3  =\frac{1}{2}\left(g^2-h\right),              \text{ and }
		f_4  =\frac{i}{2}\left(g^2+h\right)
	\end{align*}
	we get
	\begin{align}
		f & =\sum_{k=1}^4 f_{k^{2}},\notag \\
	\shortintertext{so that}
		p_f & \le\sum_{k=1}^4 p_{f_{k^{2}}}. \label{eq1654}
	\end{align}
	For each $k\in\{1,2,3,4\}$ and each $\mu\in M(G)$, we have
	\begin{equation*}
		\begin{split}
			\mu\ast f_k^2 & =\mu\circ f_k^2+(\mu\circ f_k)\ast f_k-f_k\ast (\mu\circ f_k), \\
			f_k^2\ast\mu  & =\mu\circ f_k^2+f_k\ast(\mu\circ f_k)-(\mu\circ f_k)\ast f_k,
		\end{split}
	\end{equation*}
	whence
	\begin{equation}\label{eq1655}
		p_{f_k^2}(\mu)\le 2q_{f_{k^{2}}}(\mu)+4\Vert f_k\Vert_1 q_{f_k}(\mu).
	\end{equation}
	Finally,~\eqref{eq1654} and~\eqref{eq1655} give~\eqref{eq1653}.
\end{proof}

\begin{lemma}\label{1738}
	Let $G$ and $H$ be locally compact groups, and
	let $\Phi\colon L^1(G)\to L^1(H)$ be a Jordan isomorphism.
	Then there exists a unique Jordan isomorphism $\bar{\Phi}\colon M(G)\to M(H)$
	which extends $\Phi$ and which is continuous with respect to the strict topology on both $M(G)$ and $M(H)$.
	Furthermore, $\Vert\bar{\Phi}\Vert=\Vert\Phi\Vert$.
\end{lemma}

\begin{proof}
	We begin by observing that $\Phi$ is bounded,
	because every Jordan homomorphism from a Banach algebra onto a semisimple Banach algebra is bounded (see~\cite{S}).

	Let $(e_\gamma)_{\gamma\in\Gamma}$ be an approximate identity for $L^1(G)$ of bound $1$.
	Let $\mathcal{U}$ be an ultrafilter on $\Gamma$ containing the order filter on $\Gamma$. It follows from the Banach-Alaoglu theorem that each bounded subset of $M(H)$ is relatively compact with respect to the topology $\sigma(M(H),C_0(H))$. Consequently, each bounded net
	$(\mu_\gamma)_{\gamma\in\Gamma}$ in $M(H)$ has a unique limit with respect to the topology $\sigma(M(H),C_0(H))$ along the ultrafilter $\mathcal{U}$, and we write $\lim_\mathcal{U}\mu_\gamma$ for this limit.
	It is worth noting that
	\begin{equation}\label{eq1832}
		\Bigl\Vert\lim_\mathcal{U}\mu_\gamma\Bigr\Vert \le \lim_\mathcal{U}\Vert \mu_\gamma\Vert \le \sup_{\gamma\in\Gamma}\Vert\mu_\gamma\Vert.
	\end{equation}
	Indeed, for each $\phi\in C_0(H)$ such that $\Vert \phi\Vert_\infty=1$, we have
	\[
		\left\vert\langle \mu_\gamma,\phi\rangle\right\vert 
		\le \left\Vert \mu_\gamma\right\Vert 
		\le\sup_{\gamma\in\Gamma} \left\Vert \mu_\gamma\right\Vert,
	\]
	and hence
	\[
		\Bigl\vert\Bigl\langle\lim_\mathcal{U}\mu_\gamma,\phi\Bigr\rangle\Bigr\vert =
		\lim_\mathcal{U} \vert\langle\mu_\gamma,\phi\rangle\vert \le \lim_\mathcal{U}\Vert \mu_\gamma\Vert\le
		\sup_{\gamma\in\Gamma}\Vert \mu_\gamma\Vert,
	\]
	which establishes~\eqref{eq1832}.
	Since $\mathcal{U}$ refines the order filter on $\Gamma$, we see that
	\begin{equation}\label{eq2101}
		\bigl(\mu_\gamma\bigr)_{\gamma\in\Gamma}\to \mu \ \text{ in norm }
		\implies
		\lim_\mathcal{U}\mu_\gamma=\mu.
	\end{equation}
	It should be pointed out that, for each $\nu\in M(H)$,
	\begin{equation}\label{eq2032}
		\lim_\mathcal{U}\bigl(\mu_\gamma\circ\nu\bigr)=\Bigl(\lim_\mathcal{U}\mu_\gamma\Bigr)\circ\nu.
	\end{equation}
	Indeed,
	for each $\phi\in C_0(H)$,
	\[
		\left\langle\mu_\gamma\circ\nu,\phi\right\rangle=
		\langle\mu_\gamma,\tfrac{1}{2}(\phi\cdot\nu+\nu\cdot\phi)\rangle\to
		\Bigl\langle\lim_\mathcal{U}\mu_\gamma,\tfrac{1}{2}(\phi\cdot\nu+\nu\cdot\phi)\Bigr\rangle=
		\Bigl\langle\Bigl(\lim_\mathcal{U}\mu_\gamma\Bigr)\circ\nu,\phi\Bigr\rangle.
	\]

	We now proceed to define the map $\bar{\Phi}$.
	For each $\mu\in M(G)$, we have
	\begin{equation}\label{eq1833}
		\Vert\Phi(\mu\circ e_\gamma)\Vert 
		\le \left\Vert\Phi\right\Vert \left\Vert\mu\circ e_\gamma\right\Vert
		\le \norm{\Phi} \norm{\mu} 
		\quad\forall \gamma\in\Gamma,
	\end{equation}
	and hence the net $\left(\Phi(\mu\circ e_\gamma)\right)_{\gamma\in\Gamma}$ is bounded.
	Consequently we can define a map $\bar{\Phi}\colon M(G)\to M(H)$ by
	\[
		\bar{\Phi}(\mu)=\lim_\mathcal{U}\Phi(\mu\circ e_\gamma)
		\quad\forall\mu\in M(G).
	\]

	For each $f\in L^1(G)$, we have
	\[
		\bigl(f\circ e_\gamma\bigr)_{\gamma\in\Gamma}\to f \ \text{ in norm,}
	\]
	the continuity of $\Phi$ then implies
	\[
		\bigl(\Phi(f\circ e_\gamma)\bigr)_{\gamma\in\Gamma}\to\Phi(f) \text{ in norm,}
	\]
	and~\eqref{eq2101} now shows that
	\[
		\bar{\Phi}(f)=\lim_\mathcal{U}\Phi(f \circ e_\gamma)=\Phi(f).
	\]
	Thus $\bar{\Phi}$ is an extension of $\Phi$.

	The linearity of the limit along an ultrafilter on a topological
	linear space gives the linearity of $\bar{\Phi}$.
	Take $\mu\in M(G)$.
	From~\eqref{eq1832} and~\eqref{eq1833} we deduce that
	$\Vert\bar{\Phi}(\mu)\Vert\le\Vert\Phi\Vert\Vert \mu\Vert$,
	which gives the continuity of $\bar{\Phi}$ and $\Vert\bar{\Phi}\Vert\le\Vert\Phi\Vert$.
	Since $\bar{\Phi}$ extends $\Phi$, the preceding inequality turns into $\Vert\bar{\Phi}\Vert=\Vert\Phi\Vert$.

	Our next concern will be to prove that $\bar{\Phi}$ is a Jordan homomorphism.
	The starting point of our process is the following identity,
	taken from~\cite[Identity~(B$_1$), page~33]{Jac},
	\begin{multline}\label{j1}
		\bigl(\mu_1\circ(\mu_2\circ\mu_3)\bigr)\circ\mu_4+
		\bigl(\mu_1\circ(\mu_2\circ\mu_4)\bigr)\circ\mu_3+
		\bigl(\mu_1\circ(\mu_3\circ\mu_4)\bigr)\circ\mu_2
		\\ =
		(\mu_1\circ\mu_2)\circ(\mu_3\circ\mu_4)+
		(\mu_1\circ\mu_3)\circ(\mu_2\circ\mu_4)+
		(\mu_1\circ\mu_4)\circ(\mu_2\circ\mu_3)
	\end{multline}
	for all $\mu_1,\mu_2,\mu_3,\mu_4\in M(G)$.
	Let $f,g\in L^1(G)$ and let $\mu\in M(G)$.
	We take $\mu_1=e_\gamma$ ($\gamma\in\Gamma$), $\mu_2=f$, $\mu_3=g$, and $\mu_4=\mu$ in~\eqref{j1}
	to obtain
	\begin{multline}\label{j2}
		\bigl(e_\gamma\circ(f\circ g)\bigr)\circ\mu+
		\bigl(e_\gamma\circ(f\circ\mu)\bigr)\circ g+
		\bigl(e_\gamma\circ(g\circ\mu)\bigr)\circ f
		\\ =
		(e_\gamma\circ f)\circ(g\circ\mu)+
		(e_\gamma\circ g)\circ(f\circ\mu)+
		(e_\gamma\circ\mu)\circ(f\circ g).
	\end{multline}
	We now apply $\Phi$ to both sides of~\eqref{j2} to get
	\begin{multline}\label{j22}
		\Phi\Bigl(\bigl(e_\gamma\circ(f\circ g)\bigr)\circ\mu\Bigr)+
		\Phi\Bigl(\bigl(e_\gamma\circ(f\circ\mu)\bigr)\circ g\Bigr)+
		\Phi\Bigl(\bigl(e_\gamma\circ(g\circ\mu)\bigr)\circ f\Bigr)
		\\ =
		\Phi(e_\gamma\circ f)\circ\Phi(g\circ\mu)+
		\Phi(e_\gamma\circ g)\circ\Phi(f\circ\mu) +
		\Phi(e_\gamma\circ\mu)\circ\Phi(f\circ g).
	\end{multline}
	Since
	\begin{equation*}
		\begin{split}
			\Bigl(\bigl(e_\gamma\circ(f\circ g)\bigr)\circ\mu\Bigr)_{\gamma\in\Gamma}
			 & \to (f\circ g)\circ\mu \ \text{ in norm}, \\
			\Bigl(\bigl(e_\gamma\circ(f\circ\mu)\bigr)\circ g\Bigr)_{\gamma\in\Gamma}
			 & \to (f\circ\mu)\circ g \ \text{ in norm}, \\
			\Bigl(\bigl(e_\gamma\circ(g\circ\mu)\bigr)\circ f\Bigr)_{\gamma\in\Gamma}
			 & \to (g\circ\mu)\circ f \ \text{ in norm}, \\
			\bigl(e_\gamma\circ f\bigr)_{\gamma\in\Gamma}
			 & \to f \ \text{ in norm},                  \\
			\bigl(e_\gamma\circ g\bigr)_{\gamma\in\Gamma}
			 & \to g \ \text{ in norm},
		\end{split}
	\end{equation*}
	the boundedness of $\Phi$ yields
	\begin{equation*}
		\begin{split}
			\Bigl(\Phi\Bigl(\bigl(e_\gamma\circ(f\circ g)\bigr)\circ\mu\Bigr)\Bigr)_{\gamma\in\Gamma}
			 & \to \Phi\bigl((f\circ g)\circ\mu\bigr) \ \text{ in norm}, \\
			\Bigl(\Phi\Bigl(\bigl(e_\gamma\circ(f\circ\mu)\bigr)\circ g\Bigr)\Bigr)_{\gamma\in\Gamma}
			 & \to \Phi\bigl((f\circ\mu)\circ g\bigr) \ \text{ in norm}, \\
			\Bigl(\Phi\Bigl(\bigl(e_\gamma\circ(g\circ\mu)\bigr)\circ f\Bigr)\Bigr)_{\gamma\in\Gamma}
			 & \to \Phi\bigl((g\circ\mu)\circ f\bigr) \ \text{ in norm}, \\
			\Bigl(\Phi(e_\gamma\circ f)\Bigr)_{\gamma\in\Gamma}
			 & \to\Phi(f) \ \text{ in norm},                             \\
			\Bigl(\Phi(e_\gamma\circ g)\Bigr)_{\gamma\in\Gamma}
			 & \to\Phi(g) \ \text{ in norm},
		\end{split}
	\end{equation*}
	and the last two limits also give
	\begin{equation*}
		\begin{split}
			\Bigl(\Phi(e_\gamma\circ f)\circ\Phi(g\circ\mu)\Bigr)_{\gamma\in\Gamma}
			 & \to\Phi(f)\circ\Phi(g\circ\mu) \ \text{ in norm}, \\
			\Bigl(\Phi(e_\gamma\circ g)\circ\Phi(f\circ\mu)\Bigr)_{\gamma\in\Gamma}
			 & \to\Phi(g)\circ\Phi(f\circ\mu) \ \text{ in norm}.
		\end{split}
	\end{equation*}
	From~\eqref{eq2101} it follows that
	\begin{equation*}
		\begin{split}
			\lim_{\mathcal{U}}\Phi\Bigl(\bigl(e_\gamma\circ(f\circ g)\bigr)\circ\mu\Bigr)
			 & =\Phi\bigl((f\circ g)\circ\mu\bigr), \\
			\lim_{\mathcal{U}}\Phi\Bigl(\bigl(e_\gamma\circ(f\circ\mu)\bigr)\circ g\Bigr)
			 & =\Phi\bigl((f\circ\mu)\circ g\bigr), \\
			\lim_{\mathcal{U}}\Phi\Bigl(\bigl(e_\gamma\circ(g\circ\mu)\bigr)\circ f\Bigr)
			 & =\Phi\bigl((g\circ\mu)\circ f\bigr), \\
			\lim_\mathcal{U}\Phi(e_\gamma\circ f)\circ\Phi(g\circ\mu)
			 & =\Phi(f)\circ\Phi(g\circ\mu),        \\
			\lim_\mathcal{U}\Phi(e_\gamma\circ g)\circ\Phi(f\circ\mu)
			 & =\Phi(g)\circ\Phi(f\circ\mu).
		\end{split}
	\end{equation*}
	On the other hand, by definition,
	$
		\lim_\mathcal{U}\Phi(e_\gamma\circ\mu)=\bar{\Phi}(\mu)
	$
	and hence~\eqref{eq2032} gives
	\[
		\lim_\mathcal{U}\Phi(e_\gamma\circ\mu)\circ\Phi(f\circ g)=\bar{\Phi}(\mu)\circ\Phi(f\circ g).
	\]
	Taking the limit in~\eqref{j22} along the ultrafilter $\mathcal{U}$, and using the preceding observations,
	we obtain
	\begin{multline*}\label{j3}
		\Phi\bigl((f\circ g)\circ\mu\bigr)+
		\Phi\bigl((f\circ\mu)\circ g\bigr)+
		\Phi\bigl((g\circ\mu)\circ f\bigr)
		\\ =
		\Phi(f)\circ\Phi(g\circ\mu)+
		\Phi(g)\circ\Phi(f\circ\mu)+
		\bar{\Phi}(\mu)\circ\Phi(f\circ g),
	\end{multline*}
	whence
	\begin{equation}\label{j4}
		\Phi\bigl((f\circ g)\circ\mu\bigr)
		=
		\bar{\Phi}(\mu)\circ\Phi(f\circ g).
	\end{equation}
	We now take $g=e_\gamma$ ($\gamma\in\Gamma$) in~\eqref{j4} to obtain
	\begin{equation}\label{j5}
		\Phi\bigl((f\circ e_\gamma)\circ\mu\bigr)
		=
		\bar{\Phi}(\mu)\circ\Phi(f\circ e_\gamma),
	\end{equation}
	and, taking limits in~\eqref{j5} and using the continuity of $\Phi$ and that
	\begin{equation*}
		\begin{split}
			\bigl(f\circ e_\gamma\bigr)_{\gamma\in\Gamma}
			 & \to f \ \text{ in norm,}         \\
			\bigl((f\circ e_\gamma)\circ\mu\bigr)_{\gamma\in\Gamma}
			 & \to f\circ\mu \ \text{ in norm,} \\
		\end{split}
	\end{equation*}
	we see that
	\begin{equation}\label{j6}
		\Phi(f\circ\mu)
		=
		\bar{\Phi}(\mu)\circ\Phi(f).
	\end{equation}

	We are now in a position to show that $\bar{\Phi}$ is a Jordan homomorphism.
	For this purpose we set $\mu\in M(G)$ and $g\in L^1(G)$. From~\eqref{j6} we see that, for each $\gamma\in\Gamma$,
	\[
		\Phi\bigl(\mu\circ(\mu\circ e_\gamma)\bigr)=
		\bar{\Phi}(\mu)\circ\Phi(\mu\circ e_\gamma)
	\]
	and hence that
	\begin{equation}\label{eq1047}
			\Phi\Bigl(\bigl(\mu\circ(\mu\circ e_\gamma)\bigr)\circ g\Bigr)
			  =
			\Phi\bigl(\mu\circ(\mu\circ e_\gamma)\bigr)\circ\Phi(g) 
			  =
			\bigl(\bar{\Phi}(\mu)\circ\Phi(\mu\circ e_\gamma)\bigr)\circ\Phi(g).
	\end{equation}
	Since
	\[
		\Bigl(\bigl(\mu\circ(\mu\circ e_\gamma)\bigr)\circ g\Bigr)_{\gamma\in\Gamma}\to (\mu\circ\mu)\circ g
		\ \text{ in norm},
	\]
	and $\Phi$ is continuous, we have
	\[
		\Bigl(\Phi\Bigl(\bigl(\mu\circ(\mu\circ e_\gamma)\bigr)\circ g\Bigr)\Bigr)
		\to\Phi(\mu^2\circ g)
		\ \text{ in norm.}
	\]
	On the other hand,
	$
		\lim_\mathcal{U}\Phi(\mu\circ e_\gamma)=\bar{\Phi}(\mu)
	$
	and~\eqref{eq2032} now leads to
	\[
		\lim_\mathcal{U}
		\bigl(\bar{\Phi}(\mu)\circ\Phi(\mu\circ e_\gamma)\bigr)\circ\Phi(g)=
		\bigl(\bar{\Phi}(\mu)\circ\bar{\Phi}(\mu)\bigr)\circ\Phi(g).
	\]
	We thus get
	$
		\Phi(\mu^2\circ g)=\bigl(\bar{\Phi}(\mu)\circ\bar{\Phi}(\mu)\bigr)\circ\Phi(g)
	$
	and, by~\eqref{j6},
	\[
		\bigl(\bar{\Phi}(\mu^2)-\bar{\Phi}(\mu)\circ\bar{\Phi}(\mu)\bigr)\circ\Phi(g)=0.
	\]
	The surjectivity of $\Phi$ now implies
	\[
		\bigl(\bar{\Phi}(\mu^2)-\bar{\Phi}(\mu)\circ\bar{\Phi}(\mu)\bigr)\circ h=0\quad\forall h\in L^1(H).
	\]
	We abbreviate $\bar{\Phi}(\mu^2)-\bar{\Phi}(\mu)\circ\bar{\Phi}(\mu)$ to $\nu$. What is left is to show that $\nu=0$.
	To this end we observe that, for each $h\in L^1(H)$,
	\[
		0=\bigl(\underbrace{\nu\ast e_\gamma+e_\gamma\ast\nu}_{=0})\ast h=
		\nu\ast(e_\gamma\ast h)+e_\gamma\ast(\nu\ast h)\to 2\nu\ast h
		\ \text{ in norm},
	\]
	and consequently $\nu\ast h=0$. 
	Since this holds for arbitrary $h\in L^1(H)$,
	from~\cite[Corollary 3.3.24]{D} it may be concluded that $\nu=0$, as required.

	Our next goal is the strict continuity of $\bar{\Phi}$.
	For each $h\in L^1(H)$ and each $\mu\in M(G)$, we have
	\[
		\begin{split}
			q_h \left( \bar{\Phi}(\mu)\right) & =
			\left\Vert h\circ\bar{\Phi}(\mu)\right\Vert_1 =
			\left\Vert \Phi(\Phi^{-1}(h))\circ \bar{\Phi}(\mu)\right\Vert_1 =
			\left\Vert\Phi(\Phi^{-1}(h)\circ\mu)\right\Vert_1 \\
			                                  & \le
			\left\Vert\Phi\right\Vert \left\Vert \Phi^{-1}(h)\circ\mu\right\Vert_1 =
			\left\Vert\Phi \right\Vert q_{\Phi^{-1}(h)}(\mu),
		\end{split}
	\]
	which shows that $\bar{\Phi}$ is continuous with respect to the strict topology on both $M(G)$ and $M(H)$.

	We proceed to prove that $\bar{\Phi}$ is a bijection from $M(G)$ onto $M(H)$.
	Write $\Psi=\Phi^{-1}$ and consider the map $\bar{\Psi}$. Since
	\[
		\Psi\Phi=I_{L^1(G)}, \ \Phi\Psi=I_{L^1(H)},
	\]
	$L^1(G)$ and $L^1(H)$ are dense in $M(G)$ and $M(H)$, respectively, in the strict topology, and
	both $\bar{\Phi}$ and $\bar{\Psi}$ are continuous with respect to the strict topology,
	it may be concluded that
	\[
		\bar{\Psi}\bar{\Phi}=I_{M(G)}, \ \bar{\Phi}\bar{\Psi}=I_{M(H)}.
	\]
	Hence $\bar{\Phi}$ is a Jordan isomorphism.

	Finally,
	since $L^1(G)$ is dense in $M(G)$ in the strict topology, it follows that each map $\Psi\colon M(G)\to M(H)$
	which is continuous with respect to the strict topology is uniquely specified by its values on $L^1(G)$.
	This gives the uniqueness assertion of the lemma.
\end{proof}

Throughout, $\mathbb{T}$ stands for the circle group.

\begin{theorem}\label{mainTheo}
	Let $G$ and $H$ be locally compact groups, and
	let $\Phi\colon L^1(G)\to L^1(H)$ be a contractive Jordan isomorphism.
	Then one of the following holds:
	\begin{enumerate}
		\item[\rm(i)]
		      $\Phi$ is an isometric isomorphism, and is actually expressible as
		      \[
			      \Phi f(t)=c\chi(t) f(\varphi(t))
		      \]
		      for each $f\in L^1(G)$ and almost all $t\in H$,
		      where
		      \begin{itemize}
			      \item
			            $\varphi\colon H\to G$ is a homeomorphic group isomorphism,
			      \item
			            $\chi\colon H\to\mathbb{T}$ is a continuous group homomorphism, and
			      \item
			            $c$ is the constant value of the ratio $\lambda_G(\varphi(E))/\lambda_H(E)$
			            for each measurable set $E\subseteq H$ with finite non-zero measure.
		      \end{itemize}
		\item[\rm(ii)]
		      $\Phi$ is an isometric anti-isomorphism, and is actually expressible as
		      \[
			      \Phi f(t)=c\chi(t) f(\varphi(t))\Delta_H(t^{-1})
		      \]
		      for each $f\in L^1(G)$ and almost all $t\in H$,
		      where
		      \begin{itemize}
			      \item
			            $\varphi\colon H\to G$ is a homeomorphic group anti-isomorphism,
			      \item
			            $\chi\colon H\to\mathbb{T}$ is a continuous group homomorphism,
			      \item
			            $c$ is the constant value of the ratio $\lambda_G(\varphi(E^{-1}))/\lambda_H(E)$
			            for each measurable set $E\subseteq H$ with finite non-zero measure, and
			      \item
			            $\Delta_H$ is the modular function of $H$.
		      \end{itemize}
	\end{enumerate}
	In particular, $G$ and $H$ are isomorphic as topological groups.
\end{theorem}

\begin{proof}
	Let $\bar{\Phi}\colon M(G)\to M(H)$ be the map given in Lemma~\ref{1738}.
	Our first purpose is to prove that there exist maps $\theta\colon G\to H$ and
	$\zeta\colon G\to\mathbb{T}$ uniquely specified by the condition
	\begin{equation}\label{j1003}
		\bar{\Phi}(\delta_t)=\zeta(t)\delta_{\theta(t)}
		\quad\forall t\in G.
	\end{equation}
	To this end, we will use the property that
	\[
		\mu\ast\nu\ast\mu=2\mu\circ(\mu\circ\nu)-(\mu\circ\mu)\circ\nu
		\quad\forall \mu,\nu\in M(G),
	\]
	which implies
	\begin{equation}\label{1750}
		\bar{\Phi}(\mu\ast\nu\ast\mu)=\bar{\Phi}(\mu)\ast\bar{\Phi}(\nu)\ast\bar{\Phi}(\mu)
		\quad\forall \mu,\nu\in M(G),
	\end{equation}
	because $\bar{\Phi}$ is a Jordan homomorphism.

	Set $t\in G$.
	For each $g\in L^1(H)$ we have
	\[
		\delta_{t^{-1}}\ast\delta_t\ast\Phi^{-1}(g)\ast\delta_t\ast\delta_{t^{-1}}=\Phi^{-1}(g),
	\]
	and applying $\bar{\Phi}$ to both sides of the above identity, and taking into account~\eqref{1750}, we see that
	\[
		\bar{\Phi}(\delta_{t^{-1}})\ast\bar{\Phi}(\delta_t)\ast g
		\ast\bar{\Phi}(\delta_t)\ast\bar{\Phi}(\delta_{t^{-1}})=g.
	\]
	Since $\Vert\bar{\Phi}\Vert=\Vert\Phi\Vert\le 1$, it follows that
	\[
		\begin{split}
			\left\Vert\bar{\Phi}(\delta_t)\ast g\right\Vert_1 & \le \left\Vert g\right\Vert_1
			=
			\left\Vert \bar{\Phi}(\delta_{t^{-1}})\ast\bar{\Phi}(\delta_{t})\ast g\ast\bar{\Phi}(\delta_{t})
			\ast\bar{\Phi}(\delta_{t^{-1}})\right\Vert_1                                                                 \\
			                                                  &
			\le \left\Vert \bar{\Phi}(\delta_{t^{-1}})\right\Vert \left\Vert\bar{\Phi}(\delta_{t})\ast g\right\Vert_1
			\left\Vert\bar{\Phi}(\delta_{t})\ast\bar{\Phi}(\delta_{t^{-1}})\right\Vert                                   \\
			                                                  & \le \left\Vert\bar{\Phi}(\delta_{t})\ast g\right\Vert_1,
		\end{split}
	\]
	which leads to
	\begin{equation}\label{1814}
		\left\Vert\bar{\Phi}(\delta_t)\ast g\right\Vert_1=\left\Vert g\right\Vert_1.
	\end{equation}
	Since the equation~\eqref{1814} holds true for each $g\in L^1(H)$, \cite[Theorem~3]{W2} now yields
	$\zeta(t)\in\mathbb{T}$ and $\theta(t)\in H$ such that
	\[
		\bar{\Phi}(\delta_t)=\zeta(t)\delta_{\theta(t)},
	\]
	and~\eqref{j1003} is proved.

	Our next objective is to prove that
	$\theta$ is a half-homomorphism in the sense of~\cite{Sc}, i.e.,
	\[
		\theta(st)\in\bigl\{\theta(s)\theta(t),\theta(t)\theta(s)\bigr\}
		\quad\forall s,t\in G,
	\]
	and that $\zeta$ is a group homomorphism.
	Set $s,t\in G$.
	We first observe that
	\[
		2\bar{\Phi}(\delta_s\circ\delta_t)  = \bar{\Phi}(\delta_{st}+\delta_{ts}) = \bar{\Phi}(\delta_{st})+\bar{\Phi}(\delta_{ts})
		= \zeta(st)\delta_{\theta(st)}+\zeta(ts)\delta_{\theta(ts)}
	\]
	and, being $\bar{\Phi}$ a Jordan homomorphism, we also have
	\[
		\begin{split}
			2\bar{\Phi}(\delta_s\circ\delta_t) & = 2\bar{\Phi}(\delta_s)\circ\bar{\Phi}(\delta_t) =
			2\zeta(s)\delta_{\theta(s)}\circ\zeta(t)\delta_{\theta(t)}                                                                  \\
			                                   & = \zeta(s)\zeta(t)\bigl(\delta_{\theta(s)\theta(t)}+\delta_{\theta(t)\theta(s)}\bigr).
		\end{split}
	\]
	We thus get
	\begin{equation}\label{jh1}
		\zeta(s)\zeta(t)\bigl(\delta_{\theta(s)\theta(t)}+\delta_{\theta(t)\theta(s)}\bigr)=
		\zeta(st)\delta_{\theta(st)}+\zeta(ts)\delta_{\theta(ts)}.
	\end{equation}
	We evaluate both sides of~\eqref{jh1} at the set $E=\{\theta(st)\}$, and we will divide the discussion of the
	outcome into two cases.

	(a) Suppose that $\theta(st)\neq\theta(ts)$. Then~\eqref{jh1} gives
	\[
		\zeta(s)\zeta(t)
		\bigl(\underbrace{\delta_{\theta(s)\theta(t)}(E)}_{\varepsilon}
		+\underbrace{\delta_{\theta(t)\theta(s)}(E)}_{\varepsilon'}\bigr)
		=
		\zeta(st).
	\]
	Since $\zeta(s)\zeta(t),\zeta(st)\in\mathbb{T}$ and $\varepsilon,\varepsilon'\in\{0,1\}$,
	it may be concluded that one of the following assertions hold:
	\begin{itemize}
		\item
		      $\varepsilon=1$ and $\varepsilon'=0$, which implies $\theta(st)=\theta(s)\theta(t)$ and $\zeta(st)=\zeta(s)\zeta(t)$;
		\item
		      $\varepsilon=0$ and $\varepsilon'=1$, which gives $\theta(st)=\theta(t)\theta(s)$ and $\zeta(st)=\zeta(s)\zeta(t)$.
	\end{itemize}

	(b) Suppose that $\theta(st)=\theta(ts)$. Then~\eqref{jh1} gives
	\[
		\zeta(s)\zeta(t)
		\bigl(\underbrace{\delta_{\theta(s)\theta(t)}(E)}_{\varepsilon}
		+\underbrace{\delta_{\theta(t)\theta(s)}(E)}_{\varepsilon'}\bigr)
		=
		\zeta(st)+\zeta(ts).
	\]
	\begin{itemize}
		\item
		      Assume towards a contradiction that $\varepsilon=0$.
		      Then $\theta(st)\neq\theta(s)\theta(t)$.
		      We evaluate~\eqref{jh1} at the set $E=\{\theta(s)\theta(t)\}$ to obtain
		      \[
			      \zeta(s)\zeta(t)\bigl(1+\underbrace{\delta_{\theta(t)\theta(s)}(E)}_{\in\{0,1\}}\bigr)=
			      \zeta(st)0+\zeta(ts)0=0,
		      \]
		      whence $\zeta(s)\zeta(t)=0$, a contradiction.
		\item
		      Assume that $\varepsilon'=0$. This gives $\theta(st)\neq\theta(t)\theta(s)$,
		      and we now evaluate~\eqref{jh1} at the set $E=\{\theta(t)\theta(s)\}$ to get
		      \[
			      \zeta(s)\zeta(t)\bigl(\underbrace{\delta_{\theta(s)\theta(t)}(E)}_{\in\{0,1\}}+1\bigr)=
			      \zeta(st)0+\zeta(ts)0=0.
		      \]
		      Thus $\zeta(s)\zeta(t)=0$, which is a contradiction.
		\item
		      Having ruled out in the previous points the possibility that $\varepsilon=0$ or $\varepsilon'=0$,
		      there is only the possibility that
		      $\varepsilon=\varepsilon'=1$.
		      The condition $\varepsilon=1$ implies that $\theta(st)=\theta(s)\theta(t)$,
		      and the condition $\varepsilon'=1$ yields $\theta(st)=\theta(t)\theta(s)$. Further, we have
		      \[
			      2\zeta(s)\zeta(t)=\zeta(st)+\zeta(ts).
		      \]
		      Since $\zeta(s)\zeta(t),\zeta(st),\zeta(ts)\in\mathbb{T}$, the preceding condition implies that
		      \[
			      \zeta(st)=\zeta(ts)=\zeta(s)\zeta(t).
		      \]
	\end{itemize}
	We have thus proved that
	$\zeta(st)=\zeta(s)\zeta(t)$ for all $s,t\in G$ and  that $\theta$ satisfies the property
	$\theta(st)\in\bigl\{\theta(s)\theta(t),\theta(t)\theta(s)\bigr\}$ for all $s,t\in G$,
	which, on account of~\cite[Theorem 2]{Sc}, implies that $\theta$ is either a homomorphism or an anti-homomorphism.
	We are now in a position to prove assertions (i) and (ii) according to the feature of the map $\theta$.

	(i)
	Suppose that $\theta$ is a homomorphism.
	Then, for all $s,t\in G$, we have
	\[
		\begin{split}
			\bar{\Phi}(\delta_s\ast\delta_t) & =
			\bar{\Phi}(\delta_{st})=
			\zeta(st)\delta_{\theta(st)} = \zeta(s)\zeta(t)\delta_{\theta(s)\theta(t)}                     \\
			                                 & = \zeta(s)\delta_{\theta(s)}\ast\zeta(t)\delta_{\theta(t)}=
			\bar{\Phi}(\delta_s)\ast\bar{\Phi}(\delta_t).
		\end{split}
	\]
	This implies that
	\[
		\bar{\Phi}(\mu\ast\nu)=\bar{\Phi}(\mu)\ast\bar{\Phi}(\nu)
		\quad\forall \mu,\nu\in
		\operatorname{lin}\{\delta_t : t\in G\},
	\]
	where
	$\operatorname{lin}\{\delta_t\colon t\in G\}$ stands for the linear span of the set $\{\delta_t\colon t\in G\}$.
	Since the convolution on both $M(G)$ and $M(H)$ is separately continuous with respect to the strict topology,
	the space $\operatorname{lin}\{\delta_t\colon t\in G\}$ is dense in $M(G)$ with respect to the strict topology, and
	the map $\bar{\Phi}$ is continuous with respect to the strict topology on both $M(G)$ and $M(H)$,
	it may be concluded that
	\[
		\bar{\Phi}(\mu\ast\nu)=\bar{\Phi}(\mu)\ast\bar{\Phi}(\nu)\quad\forall \mu,\nu\in M(G).
	\]
	Hence $\bar{\Phi}$ is a homomorphism and therefore so is $\Phi$.
	We already know that $\Phi$ is a contractive isomorphism and then~\cite[Theorem~5]{W2} gives assertion (i) of the theorem.
	It seems appropriate to mention in passing that
	\begin{equation}\label{e1939}
		\varphi=\theta^{-1},
		\quad
		\chi=\zeta\circ\theta^{-1}.
	\end{equation}
	Indeed, for each $t\in G$ and each $f\in L^1(G)$, we have
	\begin{equation}\label{e1940}
		\Phi(\delta_t\ast f)=\bar{\Phi}(\delta_t)\ast\Phi(f)=\zeta(t)\delta_{\theta(t)}\ast\Phi(f),
	\end{equation}
	and, on the other hand, we see that
	\[
		\begin{split}
			\Phi(\delta_t\ast f)(x) & = c\chi(x)(\delta_t\ast f)(\varphi(x)) =
			c\chi(x)f(t^{-1}\varphi(x))                                                                                                     \\
			                        & = \chi(\varphi^{-1}(t))\bigl[c\chi(\varphi^{-1}(t)^{-1}x)(f\circ\varphi)(\varphi^{-1}(t)^{-1}x)\bigr] \\
			                        & = \chi(\varphi^{-1}(t))\Phi(f)(\varphi^{-1}(t)^{-1}x)=
			\chi(\varphi^{-1}(t))\bigl(\delta_{\varphi^{-1}(t)}\ast\Phi(f)\bigr)(x)
		\end{split}
	\]
	for almost each $x\in H$,
	so that
	\begin{equation}\label{e1941}
		\Phi(\delta_t\ast f)=\chi(\varphi^{-1}(t))\delta_{\varphi^{-1}(t)}\ast\Phi(f).
	\end{equation}
	From~\eqref{e1940} and~\eqref{e1941} we deduce that $\zeta(t)\delta_{\theta(t)}=\chi(\varphi^{-1}(t))\delta_{\varphi^{-1}(t)}$
	for each $t\in G$, and~\eqref{e1939} is proved.

	(ii) Suppose that $\theta$ is an anti-homomorphism.
	Then, for all $s,t\in G$, we have
	\[
		\begin{split}
			\bar{\Phi}(\delta_s\ast\delta_t) & =
			\bar{\Phi}(\delta_{st})=
			\zeta(st)\delta_{\theta(st)}=
			\zeta(s)\zeta(t)\delta_{\theta(t)\theta(s)} \\
			                                 & =
			\zeta(t)\delta_{\theta(t)}\ast\zeta(s)\delta_{\theta(s)}=
			\bar{\Phi}(\delta_t)\ast\bar{\Phi}(\delta_s),
		\end{split}
	\]
	and, using the same arguments as in (i), we arrive at
	\[
		\bar{\Phi}(\mu\ast\nu)=\bar{\Phi}(\nu)\ast\bar{\Phi}(\mu) \quad\forall \mu,\nu\in M(G).
	\]
	Thus $\bar{\Phi}$ is an anti-homomorphism and so is $\Phi$.
	We now consider the isometric anti-automorphism $\Psi$ of $L^1(H)$ defined by
	\[
		\bigl(\Psi g\bigr)(t)=g(t^{-1})\Delta_H(t^{-1})\quad\forall g\in L^1(H), \ \forall t\in H.
	\]
	The composition $\Psi\Phi$ is a contractive isomorphism, and~\cite[Theorem~5]{W2} shows that
	it is an isometry (and hence so is $\Phi$) and it is actually expressible as
	\begin{equation}\label{e2001}
		(\Psi\Phi)f=c\chi'(f\circ\varphi')
	\end{equation}
	for each $f\in L^1(G)$, where
	$\varphi'\colon H\to G$ is a homeomorphic group isomorphism,
	$\chi'\colon H\to\mathbb{T}$ is a continuous group homomorphism, and
	$c$ is the constant value of the ratio $\lambda_G(\varphi'(E))/\lambda_H(E)$
	for each measurable set $E\subseteq H$ with finite non-zero measure.
	From~\eqref{e2001} we deduce that
	\begin{equation*}
		\Phi f(t)=c\chi'(t^{-1})f\bigl(\varphi'(t^{-1})\bigr)\Delta_H(t^{-1})
	\end{equation*}
	for each $f\in L^1(G)$ and almost all $t\in H$,
	which establishes assertion (ii) of the theorem with $\varphi\colon H\to G$ and $\chi\colon H\to\mathbb{T}$ defined by
	\[
		\chi(t)=\chi'(t^{-1}), \quad \varphi(t)=\varphi'(t^{-1})\quad\forall t\in H.
	\]
	Perhaps it is appropriate at this point to note that
	\begin{equation}\label{e2012}
		\varphi=\theta^{-1},
		\quad
		\chi=\zeta\circ\theta^{-1}.
	\end{equation}
	In order to get these equations, we observe that
	for each $t\in G$ and each $f\in L^1(G)$,
	\begin{equation}\label{e2013}
		\Phi(\delta_t\ast f)=\Phi(f)\ast\bar{\Phi}(\delta_t)=\zeta(t)\Phi(f)\ast\delta_{\theta(t)}
	\end{equation}
	and that
	\begin{multline*}
		\Phi(\delta_t\ast f)(x)  = c\chi(x)(\delta_t\ast f)(\varphi(x))\Delta_H(x^{-1})=
		c\chi(x)f(t^{-1}\varphi(x))\Delta_H(x^{-1}) \\
		= \chi(\varphi^{-1}(t))\bigl[c\chi(f\circ\varphi)\bigr](x\varphi^{-1}(t)^{-1})
		\Delta_H\bigl((x\varphi^{-1}(t)^{-1})^{-1}\bigr)\Delta_H(\varphi^{-1}(t)^{-1})
		\\
		= \chi(\varphi^{-1}(t))\Phi(f)(x\varphi^{-1}(t)^{-1})\Delta_H(\varphi^{-1}(t)^{-1}) \\
		= \chi(\varphi^{-1}(t))\bigl(\Phi(f)\ast\delta_{\varphi^{-1}(t)}\bigr)(x)
	\end{multline*}
	for almost each $x\in H$,
	so that
	\begin{equation}\label{e2014}
		\Phi(\delta_t\ast f)=\chi(\varphi^{-1}(t))\Phi(f)\ast\delta_{\varphi^{-1}(t)}.
	\end{equation}
	From~\eqref{e2013} and~\eqref{e2014} we obtain $\zeta(t)\delta_{\theta(t)}=\chi(\varphi^{-1}(t))\delta_{\varphi^{-1}(t)}$
	for each $t\in G$, which establishes~\eqref{e2012}.
\end{proof}

\section{Isometric two-sided zero product preservers}

Let $\mathcal{A}$ and $\mathcal{B}$ be Banach algebras.
We will say that a linear map $\Phi\colon\mathcal{A}\to\mathcal{B}$ is a
\emph{two-sided zero product preserver} if for all $a,b\in\mathcal{A}$
\[
	ab=ba=0 \ \implies \ \Phi(a)\Phi(b)=\Phi(b)\Phi(a)=0.
\]
The question of describing the surjective two-sided zero product preservers is addressed in
\cite{ABEV, AEV, BGV} if either $\mathcal{A}$ and $\mathcal{B}$ are $C^*$-algebras or
if $\mathcal{A}$ and $\mathcal{B}$ are group algebras.
Here we will remain in the context of group algebras, but
we now turn our attention to the case where the preservers are isometric.

\begin{theorem}
	Let $G$ and $H$ be locally compact groups, and
	let $\Phi\colon L^1(G)\to L^1(H)$ be a surjective
	isometric
	two-sided zero product preserver.
	Then $\Phi$ is expressible as
	\[
		\Phi f=\alpha \delta_x\ast\Psi f
	\]
	for each $f\in L^1(G)$, where
	\begin{itemize}
		\item
		      $\alpha\in\mathbb{T}$,
		\item
		      $x$ is an element in the centre of $H$,
		\item
		      $\Psi\colon L^1(G)\to L^1(H)$  is
		      either an isometric isomorphism or an isometric anti-isomorphism.
	\end{itemize}
	In particular, $G$ and $H$ are isomorphic as topological groups.
\end{theorem}

\begin{proof}
	The proof starts by applying~\cite[Corollary~2.7]{BGV},
	which (regardless of the isometric character of $\Phi$) shows that
	there exist a surjective continuous Jordan homomorphism $\Psi\colon L^1(G)\to L^1(H)$ and
	an invertible central measure $\mu\in M(H)$ such that
	\[
		\Phi f=\mu\ast\Psi f\quad\forall f\in L^1(G).
	\]
	Since $\Phi$ is injective, it follows that $\Psi$ is injective, and hence $\Psi$ is a Jordan isomorphism.

	Our next goal is to show that $\Psi$ is contractive.
	For each $f\in L^1(G)$ and each $g\in L^1(H)$, we have (using that $\mu$ is central)
	\[
		\begin{split}
			\Phi(f\circ\Phi^{-1}g) & =
			\mu\ast\Psi(f\circ\Phi^{-1}g)=
			\mu\ast(\Psi f\circ(\Psi\Phi^{-1}g))                             \\
			                       & = \mu\ast(\Psi f\circ(\mu^{-1}\ast g))=
			\Psi f\circ g,
		\end{split}
	\]
	and hence
	\begin{equation}\label{e1425}
		\begin{split}
			\left\Vert\Psi f\circ g\right\Vert_1
			 & = \left\Vert\Phi(f\circ\Phi^{-1}g)\right\Vert_1 =
			\left\Vert f\circ\Phi^{-1}g\right\Vert_1             \\
			 & \le
			\left\Vert f\right\Vert_1 \left\Vert\Phi^{-1} g\right\Vert_1=
			\left\Vert f \right\Vert_1 \left\Vert g\right\Vert_1.
		\end{split}
	\end{equation}
	Let $(e_\gamma)_{\gamma\in\Gamma}$ be an approximate identity for $L^1(H)$ of bound $1$.
	For each $f\in L^1(G)$, from \eqref{e1425} we see that
	\[
		\left\Vert\Psi f\circ e_\gamma \right\Vert_1 \le \left\Vert f\right\Vert_1 \quad \forall \gamma\in\Gamma,
	\]
	and taking the limit in $\gamma\in\Gamma$ we arrive at
	\[
		\left\Vert \Psi f\right\Vert_1 =
		\lim \left\Vert\Psi f\circ e_\gamma\right\Vert_1 \le \left\Vert f\right\Vert_1.
	\]
	This proves that $\Psi$ is contractive.

	We now apply Theorem~\ref{mainTheo} to conclude that $\Psi$ is either an isometric isomorphism or an isometric anti-isomorphism.

	Finally, to deal with the measure $\mu$, we note that
	\[
		\mu\ast g=\Phi\Psi^{-1}g\quad\forall g\in L^1(H),
	\]
	which gives
	\begin{equation*}
		\left\Vert\mu\ast g\right\Vert_1 = \left\Vert\Phi\Psi^{-1}g\right\Vert_1 = \left\Vert g\right\Vert_1 \quad \forall g\in L^1(H),
	\end{equation*}
	and we conclude from~\cite[Theorem 3]{W2} that $\mu=\alpha\delta_x$
	for some $\alpha\in\mathbb{C}$ with $\vert\alpha\vert=1$ and some $x\in H$. Since $\mu$ lies in the centre of $M(H)$,
	we conclude that $x$ belongs to the centre of $H$.
\end{proof}

\section{Local isometric automorphisms}

Let $\mathcal{A}$ be a Banach algebra. A linear map $\Phi\colon\mathcal{A}\to\mathcal{A}$ is called
\begin{itemize}
	\item
	      a \emph{local isometric automorphism} if for each $a\in\mathcal{A}$ there exists an isometric automorphism
	      $\Phi_a$ of $\mathcal{A}$ such that $\Phi(a)=\Phi_a(a)$,
	\item
	      an \emph{approximately local isometric automorphism} if for each $a\in\mathcal{A}$ there exists a sequence
	      $(\Phi_{a,n})$ of isometric automorphisms of $\mathcal{A}$ such that $\Phi(a)=\lim\Phi_{a,n}(a)$.
\end{itemize}

\begin{lemma}\label{mainlemma}
	Let $G$ be a unimodular locally compact group, and
	let $\Phi\colon L^1(G)\to L^1(G)$ be a surjective Jordan homomorphism.
	Suppose that $\Phi$ is an approximately local isometric automorphism.
	Then $\Phi$ is an isometric automorphism.
\end{lemma}

\begin{proof}
	We begin by proving that $\Phi$ is an isometry. Let $f\in L^1(G)$ and take a sequence $(\Phi_{f,n})$ of
	isometric automorphisms of $L^1(G)$ such that $\Phi(f)=\lim\Phi_{f,n}(f)$. Then
	\[
		\left\Vert\Phi(f)\right\Vert_1 = \lim \left\Vert\Phi_{f,n}(f)\right\Vert_1 = \left\Vert f\right\Vert_1.
	\]
	Since, by hypothesis, $\Phi$ is surjective, we conclude that $\Phi$ is an isometric Jordan automorphism.

	Theorem~\ref{mainTheo} shows that $\Phi$ is either an isometric automorphism, as claimed, or an isometric anti-automorphism. In this latter case, we will show that $G$ is abelian, which clearly implies
	that $\Phi$ is also an automorphism.
	From now on we assume that $\Phi$ is an anti-automorphism, so that it is actually expressible as
	\[
		\Phi(f)=c\chi (f\circ\varphi)
		\quad
		\forall f\in L^1(G),
	\]
	where
	\begin{itemize}
		\item
		      $\varphi\colon G\to G$ is a homeomorphic group anti-automorphism,
		\item
		      $\chi\colon G\to\mathbb{T}$ is a continuous group homomorphism,
		\item
		      $c$ is the constant value of the ratio $\lambda_G(\varphi(E^{-1}))/\lambda_G(E)$ for each measurable set $E\subseteq G$
		      with finite non-zero measure.
	\end{itemize}
	It should be pointed out that the modular function has been omitted in the above representation
	because $G$ is unimodular.
	Our goal is to show that $G$ is abelian.

	Assume towards a contradiction that there exist $s,t\in G$ such that $st\neq ts$.
	Then the elements $s,t,st,ts$ are mutually distinct and hence there exists
	a compact neighbourhood $V$ of the identity $e$ of $G$ such that
	the sets
	\begin{equation}\label{eq13}
		sV,tV,(st)V,(ts)V
		\text{ are mutually disjoint.}
	\end{equation}
	Further, we take a compact neighbourhood $U$ of $e$ such tat
	\begin{equation*}
		U\subseteq V \text{ and } (sU)(tU)\subseteq (st) V.
	\end{equation*}
	Consequently, the sets
	\begin{equation*}
		sU, tU, stV, tsV
		\text{ are mutually disjoint.}
	\end{equation*}
	We can define $g\in L^1(G)$ by
	\[
		g=\mathbf{1}_{sU}+2\mathbf{1}_{tU}+3\mathbf{1}_{stV},
	\]
	where $\mathbf{1}_{E}$ denotes the characteristic function of $E$
	for each $E\subseteq G$.

	It follows from the hypothesis that there exists
	a sequence $(\Phi_{n})$ of isometric automorphisms of $L^1(G)$ such that
	\begin{equation*}
		\Phi(g)=\lim_{n\to\infty}\Phi_n(g).
	\end{equation*}
	Of course, for each $n\in\mathbb{N}$, $\Phi_n$ is expressible as
	\[
		\Phi_n(f)=c_n\chi_n (f\circ\varphi_n)
		\quad
		\forall f\in L^1(G),
	\]
	where
	\begin{itemize}
		\item
		      $\varphi_n\colon G\to G$ is a homeomorphic group automorphism,
		\item
		      $\chi_n\colon G\to\mathbb{T}$ is a continuous group homomorphism,
		\item
		      $c_n$ is the constant value of the ratio $\lambda_G(\varphi_n(E))/\lambda_G(E)$ for each measurable set $E\subseteq G$ with finite non-zero measure.
	\end{itemize}
	We thus get
	\begin{equation*}
		\lim_{n\to\infty} \left\Vert c\chi(g\circ\varphi)-c_n\chi_n(g\circ\varphi_n)\right\Vert_1 = 0,
	\end{equation*}
	and, since
	\[
		\bigl\Vert \vert c\chi(g\circ\varphi)\vert-\vert c_n\chi_n(g\circ\varphi_n)\vert\bigr\Vert_1
		\le
		\left\Vert c\chi(g\circ\varphi)-c_n\chi_n(g\circ\varphi_n) \right\Vert_1
		\quad\forall n\in\mathbb{N},
	\]
	we deduce that
	\begin{equation}\label{eq14}
		\lim_{n\to\infty} \left\Vert c g\circ\varphi-c_n g\circ\varphi_n\right\Vert_1 = 0.
	\end{equation}
	By passing to a subsequence, we may suppose that
	\begin{equation}\label{eq14b}
		\lim_{n\to\infty}\bigl( c g\circ\varphi-c_n g\circ\varphi_n\bigr)=0
		\quad\text{almost everywhere on $G$.}
	\end{equation}
	We proceed to show that
	\begin{equation}\label{eq14c}
		\lim_{n\to\infty}c_n=c.
	\end{equation}
	Note that
	\begin{equation}\label{eq1750}
		\begin{split}
			g\circ\varphi   & =
			\mathbf{1}_{\varphi^{-1}(sU)}+2\mathbf{1}_{\varphi^{-1}(tU)}+3\mathbf{1}_{\varphi^{-1}(stV)}, \\
			g\circ\varphi_n & =
			\mathbf{1}_{\varphi_n^{-1}(sU)}+2\mathbf{1}_{\varphi_n^{-1}(tU)}+3\mathbf{1}_{\varphi_n^{-1}(stV)}
		\end{split}
	\end{equation}
	and, since $\lambda_G\bigl(\varphi^{-1}(sU)\bigr),\lambda_G\bigl(\varphi^{-1}(tU)\bigr)\ne 0$,
	on account of~\eqref{eq14b} we can take $x\in \varphi^{-1}(sU)$ and $y\in\varphi^{-1}(tU)$
	such that
	\begin{multline}\label{eq14d}
		\lim_{n\to\infty}
		c_n\bigl(\mathbf{1}_{\varphi_n^{-1}(sU)}(x)+2\mathbf{1}_{\varphi_n^{-1}(tU)}(x)+3\mathbf{1}_{\varphi_n^{-1}(stV)}(x)\bigr)\\
		= c\bigl(\mathbf{1}_{\varphi^{-1}(sU)}(x)+2\mathbf{1}_{\varphi^{-1}(tU)}(x)+3\mathbf{1}_{\varphi^{-1}(stV)}(x)\bigr)
	\end{multline}
	and
	\begin{multline}\label{eq14e}
		\lim_{n\to\infty}
		c_n\bigl(\mathbf{1}_{\varphi_n^{-1}(sU)}(y)+2\mathbf{1}_{\varphi_n^{-1}(tU)}(y)+3\mathbf{1}_{\varphi_n^{-1}(stV)}(y)\bigr)\\
		= c\bigl(\mathbf{1}_{\varphi^{-1}(sU)}(y)+2\mathbf{1}_{\varphi^{-1}(tU)}(y)+3\mathbf{1}_{\varphi^{-1}(stV)}(y)\bigr).
	\end{multline}
	By~\eqref{eq13},
	\begin{gather*}
		p_n:=
		\mathbf{1}_{\varphi_n^{-1}(sU)}(x)+2\mathbf{1}_{\varphi_n^{-1}(tU)}(x)+3\mathbf{1}_{\varphi_n^{-1}(stV)}(x)\in\{0,1,2,3\}, \\
		q_n:=
		\mathbf{1}_{\varphi_n^{-1}(sU)}(y)+2\mathbf{1}_{\varphi_n^{-1}(tU)}(y)+3\mathbf{1}_{\varphi_n^{-1}(stV)}(y)\in\{0,1,2,3\}, \\
		\mathbf{1}_{\varphi^{-1}(sU)}(x)+2\mathbf{1}_{\varphi^{-1}(tU)}(x)+3\mathbf{1}_{\varphi^{-1}(stV)}(x)
		=1, \\
		\shortintertext{and}
		\mathbf{1}_{\varphi^{-1}(sU)}(y)+2\mathbf{1}_{\varphi^{-1}(tU)}(y)+3\mathbf{1}_{\varphi^{-1}(stV)}(y)=
		2,
	\end{gather*}
	and, by~\eqref{eq14d},~\eqref{eq14e}, we have
	\begin{equation}\label{eq14f}
		\lim_{n\to\infty}c_np_n=c,
		\quad
		\lim_{n\to\infty}c_nq_n=2c.
	\end{equation}
	This clearly forces that  $p_n,q_n\ne 0$ for sufficiently large $n\in\mathbb{N}$
	and
	\[
		\lim_{n\to\infty}\frac{p_n}{q_n}=\frac{1}{2},
	\]
	which, in turns, implies that
	$
		\frac{p_n}{q_n}=\frac{1}{2}
	$
	for sufficiently large $n\in\mathbb{N}$,
	and hence that $p_n=1$ and $q_n=2$
	for sufficiently large $n\in\mathbb{N}$.
	Now~\eqref{eq14f} yields~\eqref{eq14c}.

	From~\eqref{eq14} and~\eqref{eq14c} we see that
	\begin{equation}\label{1754}
		\lim_{n\to\infty} \left\Vert g\circ\varphi-g\circ\varphi_n\right\Vert_1=0.
	\end{equation}

	Let $W$ be a neighbourhood of $e$ such that $W\subseteq U$.
	We claim that
	\begin{align}
		\lim_{n\to\infty}
		\lambda_G\bigl(\varphi_n^{-1}(sW)\cap\varphi^{-1}(sU)\bigr) & =
		c^{-1}\lambda_G(W), \label{eq2100}                                                                   \\
		\lim_{n\to\infty}
		\lambda_G\bigl(\varphi_n^{-1}(tW)\cap\varphi^{-1}(tU)\bigr) & =
		c^{-1}\lambda_G(W), \label{eq2200}                                                                   \\
		\shortintertext{and}
		\liminf_{n\to\infty}
		\lambda_G\bigl(E_n\cap\varphi_n^{-1}(stV)\bigr)             & \ge c^{-1}\lambda_G(W), \label{eq2300}
	\end{align}
	where
	\[
		E_n=\bigl(\varphi_n^{-1}(tW)\cap\varphi^{-1}(tU)\bigr)\bigl(\varphi_n^{-1}(sW)\cap\varphi^{-1}(sU)\bigr)
		\quad\forall n\in\mathbb{N}.
	\]
	In order to prove~\eqref{eq2100}, write
	\[
		\begin{split}
			A_n & =\varphi_n^{-1}(sW)\cap\varphi^{-1}(sU),                                                                           \\
			B_n & =\varphi_n^{-1}(sW)\cap\varphi^{-1}(tU),                                                                           \\
			C_n & =\varphi_n^{-1}(sW)\cap\varphi^{-1}(stV),                                                                          \\
			D_n & =\varphi_n^{-1}(sW)\cap\bigr[G\setminus\bigl(\varphi^{-1}(sU)\cup\varphi^{-1}(tU)\cup\varphi^{-1}(stV)\bigr)\bigr]
		\end{split}
	\]
	for each $n\in\mathbb{N}$, and note that
	\[
		\varphi_n^{-1}(sW)=A_n\cup B_n\cup C_n\cup D_n.
	\]
	From~\eqref{eq1750} we deduce that
	\[
		\begin{split}
			\mathbf{1}_{\varphi_n^{-1}(sW)}(g\circ\varphi-g\circ\varphi_n) & =
			\mathbf{1}_{A_n}+2\mathbf{1}_{B_n}+3\mathbf{1}_{C_n}-\mathbf{1}_{\varphi_n^{-1}(sW)}                                                  \\
			                                                               & =
			\mathbf{1}_{A_n}+2\mathbf{1}_{B_n}+3\mathbf{1}_{C_n} -\bigl(\mathbf{1}_{A_n}+\mathbf{1}_{B_n}+\mathbf{1}_{C_n}+\mathbf{1}_{D_n}\bigr) \\
			                                                               & = \mathbf{1}_{B_n}+2\mathbf{1}_{C_n}-\mathbf{1}_{D_n},
		\end{split}
	\]
	whence
	\[
		\bigl\Vert\mathbf{1}_{\varphi_n^{-1}(sW)}(g\circ\varphi-g\circ\varphi_n)\bigr\Vert_1=
		\lambda_G(B_n)+2\lambda_G(C_n)+\lambda_G(D_n).
	\]
	Since
	\[
		\bigl\Vert\mathbf{1}_{\varphi_n^{-1}(sW)}(g\circ\varphi-g\circ\varphi_n)\bigr\Vert_1\le
		\bigl\Vert g\circ\varphi-g\circ\varphi_n\bigr\Vert_1,
	\]
	\eqref{1754} then gives
	\[
		\lim_{n\to\infty}\bigl(\lambda_G(B_n)+2\lambda_G(C_n)+\lambda_G(D_n)\bigr)=0,
	\]
	so that
	\[
		\lim_{n\to\infty}\lambda_G(B_n)=
		\lim_{n\to\infty}\lambda_G(C_n)=
		\lim_{n\to\infty}\lambda_G(D_n)=0.
	\]
	We now see that
	\[
		\lambda_G(\varphi_n^{-1}(sW))=\lambda_G(A_n)+
		\underbrace{\lambda_G(B_n)+\lambda_G(C_n)+\lambda_G(D_n)}_{\to 0}
	\]
	and
	\[
		\lambda_G(\varphi_n^{-1}(sW))=c_n^{-1}\lambda_G(sW)=c_n^{-1}\lambda_G(W)\to c^{-1}\lambda_G(W),
	\]
	and we thus obtain~\eqref{eq2100}.
	We can apply the same arguments as before, with $\varphi_n^{-1}(sW)$ replaced by $\varphi_n^{-1}(tW)$,
	to obtain~\eqref{eq2200}.
	We take
	\[
		\begin{split}
			A_n & =\varphi_n^{-1}(tW)\cap\varphi^{-1}(sU),                                                                           \\
			B_n & =\varphi_n^{-1}(tW)\cap\varphi^{-1}(tU),                                                                           \\
			C_n & =\varphi_n^{-1}(tW)\cap\varphi^{-1}(stV),                                                                          \\
			D_n & =\varphi_n^{-1}(tW)\cap\bigr[G\setminus\bigl(\varphi^{-1}(sU)\cup\varphi^{-1}(tU)\cup\varphi^{-1}(stV)\bigr)\bigr]
		\end{split}
	\]
	for each $n\in\mathbb{N}$, and note that
	\[
		\varphi_n^{-1}(tW)=A_n\cup B_n\cup C_n\cup D_n.
	\]
	From~\eqref{eq1750} we see that
	\[
		\begin{split}
			\mathbf{1}_{\varphi_n^{-1}(tW)}(g\circ\varphi-g\circ\varphi_n) & =
			\mathbf{1}_{A_n}+2\mathbf{1}_{B_n}+3\mathbf{1}_{C_n}-2\mathbf{1}_{\varphi_n^{-1}(sW)}                                                            \\
			                                                               & =
			\mathbf{1}_{A_n}+2\mathbf{1}_{B_n}+3\mathbf{1}_{C_n}  \quad {} -2\bigl(\mathbf{1}_{A_n}+\mathbf{1}_{B_n}+\mathbf{1}_{C_n}+\mathbf{1}_{D_n}\bigr) \\
			                                                               & =
			-\mathbf{1}_{A_n}+\mathbf{1}_{C_n}-2\mathbf{1}_{D_n},
		\end{split}
	\]
	and therefore
	\[
		\bigl\Vert\mathbf{1}_{\varphi_n^{-1}(tW)}(g\circ\varphi-g\circ\varphi_n)\bigr\Vert_1=
		\lambda_G(A_n)+\lambda_G(C_n)+2\lambda_G(D_n).
	\]
	By using that
	\[
		\bigl\Vert\mathbf{1}_{\varphi_n^{-1}(sW)}(g\circ\varphi-g\circ\varphi_n)\bigr\Vert_1\le
		\bigl\Vert g\circ\varphi-g\circ\varphi_n\bigr\Vert_1
	\]
	and
	\eqref{1754} we obtain
	\[
		\lim_{n\to\infty}\bigl(\lambda_G(A_n)+\lambda_G(C_n)+2\lambda_G(D_n)\bigr)=0,
	\]
	so that
	\[
		\lim_{n\to\infty}\lambda_G(A_n)=
		\lim_{n\to\infty}\lambda_G(C_n)=
		\lim_{n\to\infty}\lambda_G(D_n)=0.
	\]
	Since
	\[
		\lambda_G(\varphi_n^{-1}(tW))=\lambda_G(B_n)+
		\underbrace{\lambda_G(A_n)+\lambda_G(C_n)+\lambda_G(D_n)}_{\to 0},
	\]
	we conclude that
	\[
		\lambda_G(\varphi_n^{-1}(tW))=c_n^{-1}\lambda_G(tW)=c_n^{-1}\lambda_G(W)\to c^{-1}\lambda_G(W),
	\]
	and we thus obtain~\eqref{eq2200}.
	Our next concern will be~\eqref{eq2300}. To this end, set
	\[
		\begin{split}
			A_n & =E_n\cap\varphi_n^{-1}(sW),                                                                               \\
			B_n & =E_n\cap\varphi_n^{-1}(tW),                                                                               \\
			C_n & =E_n\cap\varphi_n^{-1}(stV),                                                                              \\
			D_n & =E_n\cap\bigr[G\setminus\bigl(\varphi_n^{-1}(sU)\cup\varphi_n^{-1}(tU)\cup\varphi_n^{-1}(stV)\bigr)\bigr]
		\end{split}
	\]
	for each $n\in\mathbb{N}$, and note that
	\[
		E_n=A_n\cup B_n\cup C_n\cup D_n
	\]
	and that
	\[
		E_n\subseteq
		\varphi^{-1}(tU)\varphi^{-1}(sU)=\varphi^{-1}(sUtU)\subseteq\varphi^{-1}(stV).
	\]
	From~\eqref{eq1750} we conclude that
	\begin{align*}
		\mathbf{1}_{E_n}(g\circ\varphi-g\circ\varphi_n) & =
		3\mathbf{1}_{E_n}-\bigl(\mathbf{1}_{A_n}+2\mathbf{1}_{B_n}+3\mathbf{1}_{C_n}\bigr)                                                    \\
		                                                & = 3\bigl(\mathbf{1}_{A_n}+\mathbf{1}_{B_n}+\mathbf{1}_{C_n}+\mathbf{1}_{D_n}\bigr)-
		\bigl(\mathbf{1}_{A_n}+2\mathbf{1}_{B_n}+3\mathbf{1}_{C_n}\bigr)                                                                      \\
		                                                & = 2\mathbf{1}_{A_n}+\mathbf{1}_{B_n}+3\mathbf{1}_{D_n},
	\end{align*}
	hence that
	\[
		\bigl\Vert\mathbf{1}_{E_n}(g\circ\varphi-g\circ\varphi_n)\bigr\Vert_1=
		2\lambda_G(A_n)+\lambda_G(B_n)+3\lambda_G(D_n),
	\]
	and finally that
	\[
		\lim_{n\to\infty}\bigl(2\lambda_G(A_n)+\lambda_G(B_n)+3\lambda_G(D_n)\bigr)=0,
	\]
	because
	\[
		\bigl\Vert\mathbf{1}_{E_n}(g\circ\varphi-g\circ\varphi_n)\bigr\Vert_1\le
		\bigl\Vert g\circ\varphi-g\circ\varphi_n\bigr\Vert_1\to 0.
	\]
	We thus obtain
	\[
		\lim_{n\to\infty}\lambda_G(A_n)=
		\lim_{n\to\infty}\lambda_G(B_n)=
		\lim_{n\to\infty}\lambda_G(D_n)=0.
	\]
	On the other hand, we have
	\begin{equation*}
		\lambda_G(E_n)=\lambda_G(C_n)+\underbrace{\lambda_G(A_n)+\lambda_G(B_n)+\lambda_G(D_n)}_{\to 0},
	\end{equation*}
	which implies
	\begin{equation}\label{eq946}
		\liminf_{n\to\infty}\lambda_G(E_n)=
		\liminf_{n\to\infty}\lambda_G(C_n).
	\end{equation}
	On account of~\eqref{eq2100} and~\eqref{eq2200},
	the sets $\varphi_n^{-1}(sW)\cap\varphi^{-1}(sU)$ and $\varphi_n^{-1}(tW)\cap\varphi^{-1}(tU)$ are
	non-empty for sufficiently large $n\in\mathbb{N}$, so that by choosing $x_n\in\varphi_n^{-1}(tW)\cap\varphi^{-1}(tU)$
	we get
	$x_n\bigl(\varphi_n^{-1}(sW)\cap\varphi^{-1}(sU)\bigr)\subset E_n$
	and hence
	\[
		\lambda_G\bigl(\varphi_n^{-1}(sW)\cap\varphi^{-1}(sU)\bigr)
		=
		\lambda_G\bigl(x_n\bigl(\varphi_n^{-1}(sW)\cap\varphi^{-1}(sU)\bigr)\bigr)
		\le
		\lambda_G(E_n).
	\]
	Using~\eqref{eq2100}, this leads to
	\[
		c^{-1}\lambda_G(W)=\lim_{n\to\infty}\lambda_G\bigl(\varphi_n^{-1}(sW)\cap\varphi^{-1}(sU)\bigr)\le
		\liminf_{n\to\infty}\lambda_G(E_n),
	\]
	which (combined with~\eqref{eq946}) gives
	\[
		c^{-1}\lambda_G(W)\le\liminf_{n\to\infty}\lambda_G(C_n),
	\]
	and this is precisely the assertion~\eqref{eq2300}.
	The crucial information given by~\eqref{eq2300} about the set
	\[
		\Bigl[\bigl(\varphi_n^{-1}(tW)\cap\varphi^{-1}(tU)\bigr)\bigl(\varphi_n^{-1}(sW)\cap\varphi^{-1}(sU)\bigr)\Bigr]
		\cap\varphi_n^{-1}(stV)
	\]
	is that it is non-empty for sufficiently large $n\in\mathbb{N}$. Fix such a $n\in\mathbb{N}$
	and then take $x_W,y_W\in W$ such that
	\begin{equation*}
		\begin{split}
			\varphi_n^{-1}(sx_W)                     & \in\varphi^{-1}(sU),    \\
			\varphi_n^{-1}(ty_W)                     & \in\varphi^{-1}(tU),    \\
			\varphi_n^{-1}(ty_W)\varphi_n^{-1}(sx_W) & \in\varphi_n^{-1}(stV).
		\end{split}
	\end{equation*}
	Now the last equation gives
	\begin{equation}\label{eq1033}
		ty_Wsx_W\in stV.
	\end{equation}

	We are now in a position to proceed with the final step of the proof.
	Let $\mathcal{W}$ be the family of neighbourhoods $W$ of $e$ such that $W\subseteq U$.
	Then $\mathcal{W}$ is a directed set with the order defined by $W_1\le W_2$
	if $W_2\subseteq W_1$.
	For each $W\in\mathcal{W}$ we choose $x_W,y_W\in W$ satisfying condition~\eqref{eq1033}.
	We now observe that
	\[
		(x_W)_{W\in\mathcal{W}}\to e, \ (y_W)_{W\in\mathcal{W}}\to e,
	\]
	so that
	\[
		(ty_Wsx_W)_{W\in\mathcal{W}}\to ts.
	\]
	Since the set $stV$ is compact, from~\eqref{eq1033} it may be concluded that
	$
		ts\in  stV,
	$
	which contradicts~\eqref{eq13}.
\end{proof}

\begin{theorem}\label{mainTh4}
	Let $G$ be a unimodular locally compact group, and
	let $\Phi\colon L^1(G)\to L^1(G)$ be a surjective local isometric automorphism.
	Then $\Phi$ is an isometric automorphism.
\end{theorem}

\begin{proof}
	On account of Lemma~\ref{mainlemma}, we only need to prove that $\Phi$ is a Jordan homomorphism.

	Let $C_{00}(G)$ be the subalgebra of $L^1(G)$ consisting of continuous functions on $G$ with compact support.
	Since $C_{00}(G)$ is dense in $L^1(G)$ and $\Phi$ is continuous (actually $\Phi$ is an isometry), it suffices
	to prove that $\Phi$ is a Jordan homomorphism on $C_{00}(G)$.
	Our next arguments for this task are inspired by ideas from~\cite{ABEGV}.

	Let us consider the evaluation functional $\mathcal{E}\colon C_{00}(G)\to\mathbb{C}$ defined by
	\[
		\mathcal{E}(f)=f(e)\quad\forall f\in C_{00}(G).
	\]
	We observe that $\mathcal{E}$ has the tracial property
	\begin{equation}\label{e747}
		\mathcal{E}(f\ast g)=\mathcal{E}(g\ast f)\quad\forall f,g\in C_{00}(G).
	\end{equation}
	Indeed, using the unimodularity of $G$, we see that
	\[
		\begin{split}
			\mathcal{E}(f\ast g) & = \int_{G}f(t)g(t^{-1}e)\,d\lambda_G(t) = \int_{G}f(t)g(t^{-1})\,d\lambda_G(t)                        \\
			                     & = \int_{G}f(t^{-1})g(t)\,d\lambda_G(t) = \int_{G}g(t)f(t^{-1}e)\,d\lambda_G(t) = \mathcal{E}(g\ast f)
		\end{split}
	\]
	for all $f,g\in C_{00}(G)$.

	We now proceed to show that $\Phi$ maps $C_{00}(G)$ onto itself and that
	there exists a constant $c\ne 0$ such that
	\begin{equation}\label{e748}
		\mathcal{E}\bigl(\Phi f\bigr)=c\mathcal{E}(f)\quad\forall f\in C_{00}(G).
	\end{equation}
	Let $f\in C_{00}(G)$, and take an isometric automorphism $\Phi_f$ of $L^1(G)$ such that
	$
		\Phi f=\Phi_f f
	$.
	We know that $\Phi_f$ is expressible as
	\[
		\Phi _f g(t)=c_f\chi_f(t) g(\varphi_f(t))
	\]
	for each $g\in L^1(G)$ and almost all $t\in H$, where
	\begin{itemize}
		\item
		      $\varphi_f\colon G\to G$ is a homeomorphic group isomorphism,
		\item
		      $\chi_f\colon G\to\mathbb{T}$ is a continuous group homomorphism, and
		\item
		      $c_f$ is the constant value of the ratio $\lambda_G(\varphi_f(E))/\lambda_G(E)$
		      for each measurable set $E\subseteq G$ with finite non-zero measure.
	\end{itemize}
	We will continue to use the above notation throughout this proof.
	Consequently, we have
	\[
		\Phi f=c_f\chi_f (f\circ\varphi_f)\in C_{00}(G)
	\]
	and
	\begin{equation}\label{e829}
		\mathcal{E}\bigl(\Phi f\bigr)=
		c_f\underbrace{\chi_f(e)}_{=1}f(\underbrace{\varphi_f(e)}_{=e})=c_f\mathcal{E}(f).
	\end{equation}
	On account of~\eqref{e829},
	the linear functional $f\mapsto \mathcal{E}(\Phi f)$ on $C_{00}(G)$ has the property
	\[
		f\in  C_{00}(G), \ \mathcal{E}(\Phi f)=0 \ \iff \ \mathcal{E}(f)=0
	\]
	and this implies that there exists a constant $c\ne 0$ such that~\eqref{e748} holds.
	It should be pointed out that
	\begin{equation}\label{930}
		f\in C_{00}(G), \ \mathcal{E}(f)\ne 0  \implies  c_f=c.
	\end{equation}
	We now check that $\Phi(C_{00}(G))=C_{00}(G)$.
	Set $h\in C_{00}(G)$, and let $f=\Phi^{-1}h$.
	Then $h=c_f\chi_f(f\circ\varphi_f)$, and hence
	\[
		f=c_{f}^{-1}\frac{1}{\chi_f\circ\varphi_f^{-1}}\, h\circ\varphi_{f}^{-1}\in C_{00}(G).
	\]

	We next claim that
	\begin{align}
		\mathcal{E}(\Phi f\ast\Phi f)           & =c\mathcal{E}(f\ast f), \label{850}      \\
		\mathcal{E}(\Phi f\ast\Phi f\ast\Phi f) & =c\mathcal{E}(f\ast f\ast f) \label{851}
	\end{align}
	for each $f\in C_{00}(G)$.
	Given $f\in C_{00}(G)$, we have
	\[
		\begin{split}
			\Phi f\ast\Phi f           & =\Phi_f f\ast\Phi_f f=\Phi_f(f\ast f),                   \\
			\Phi f\ast\Phi f\ast\Phi f & =\Phi_f f\ast\Phi_f f\ast\Phi_f f=\Phi_f(f\ast f\ast f),
		\end{split}
	\]
	and therefore
	\[
		\begin{split}
			\mathcal{E}\bigl(\Phi f\ast\Phi f\bigr)
			 & =\mathcal{E}\bigl(\Phi_f(f\ast f)\bigr)=c_f\mathcal{E}(f\ast f),              \\
			\mathcal{E}\bigl(\Phi f\ast\Phi f\ast\Phi f\bigr)
			 & = \mathcal{E}\bigl(\Phi_f(f\ast f\ast f)\bigr)=c_f\mathcal{E}(f\ast f\ast f).
		\end{split}
	\]
	If $f\in C_{00}(G)$ is such that $\mathcal{E}(f)\ne 0$, then, using~\eqref{930}, we obtain both~\eqref{850} and~\eqref{851}.
	If $f\in C_{00}(G)$ is such that $\mathcal{E}(f)=0$, then we fix $g\in C_{00}(G)$ with $\mathcal{E}(g)\ne 0$, using
	what has already been proved with $f$ replaced by $f+\alpha g$ ($\alpha\in\mathbb{C}$), and letting $\alpha\to 0$
	we also arrive at~\eqref{850} and~\eqref{851}.
	Now,
	for all $f,g\in C_{00}(G)$ and $\alpha\in\mathbb{C}$, we replace $f$ by $f+\alpha g$ in~\eqref{850} and
	\eqref{851} and then we identify the coefficients of $\alpha$ in both sides of the corresponding equations
	to obtain
	\begin{equation*}
		\begin{split}
			\mathcal{E}(\Phi f\ast\Phi g+\Phi g\ast\Phi f) & =
			c\mathcal{E}(f\ast g+g\ast f)
		\end{split}
	\end{equation*}
	and
	\begin{gather*}
		\mathcal{E}\left(\Phi f\ast\Phi f\ast\Phi g+\Phi f\ast\Phi g\ast\Phi f + \Phi g \ast \Phi f \ast \Phi f\right)\\
		=c\mathcal{E}\left(f\ast f\ast g+f\ast g\ast f+g\ast f\ast f\right).
	\end{gather*}
	From~\eqref{e747} we conclude that
	\begin{equation}\label{959}
		\begin{split}
			\mathcal{E} \left(\Phi f\ast\Phi g \right)          & =
			c\mathcal{E} \left(f\ast g\right),                      \\
			\mathcal{E} \left(\Phi f\ast\Phi f\ast\Phi g\right) & =
			c\mathcal{E}\left(f\ast f\ast g \right).
		\end{split}
	\end{equation}

	We can now address the question of proving that $\Phi$ is a Jordan homomorphism on $C_{00}(G)$.
	Set $f\in C_{00}(G)$, and let $F=\Phi(f\ast f)-\Phi f\ast\Phi f$.
	For each $g\in L^1(G)$, we use~\eqref{959} to see that
	\[
		\mathcal{E} \bigl(\Phi(f\ast f)\ast\Phi g\bigr)=
		c\mathcal{E}\left((f\ast f)\ast g\right)=\mathcal{E}\bigl(\Phi f\ast\Phi f\ast\Phi g\bigr),
	\]
	so that
	$
		\mathcal{E}\bigl(F\ast\Phi g\bigr)=0.
	$
	Since $\Phi(C_{00}(G))=C_{00}(G)$, it follows that
	$
		\mathcal{E}\bigl(F\ast h\bigr)=0
	$
	for each $h\in C_{00}(G)$. By defining $h\in C_{00}(G)$ by
	\[
		h(t)=\overline{F(t^{-1})}\quad\forall t\in G
	\]
	we arrive at
	\[
		\begin{split}
			0 & =\mathcal{E} \left(F\ast h\right) = \int_G F(t)h(t^{-1}e)\,d\lambda_G(t)            \\
			  & = \int_G F(t)\overline{F(t)}\,d\lambda_G(t)=\int_G\vert F(t)\vert^2\,d\lambda_G(t),
		\end{split}
	\]
	which gives $F=0$, and hence $\Phi(f\ast f)=\Phi f\ast\Phi f$, as required.
\end{proof}

If $G$ is a discrete group, then the Haar measure is the counting measure, and
the corresponding group algebra will be written as $\ell^1(G)$.

\begin{theorem}\label{mainTh3}
	Let $G$ be a discrete group, and
	let $\Phi\colon \ell^1(G)\to\ell^1(G)$ be a surjective approximately local isometric automorphism.
	Then $\Phi$ is an isometric automorphism.
\end{theorem}

\begin{proof}
	By Lemma~\ref{mainlemma}, it suffices to show that $\Phi$ is a Jordan homomorphism.
	For this purpose we will use a method similar to that used in Theorem~\ref{mainTh4}.

	Let us define a linear functional $\mathcal{E}\colon\ell^1(G)\to\mathbb{C}$ by
	\[
		\mathcal{E}(f)=f(e)\quad\forall f \in \ell^1(G).
	\]
	We observe that
	\[
		\left\vert \mathcal{E}(f)\right\vert \le \left\Vert f\right\Vert_1 \quad\forall f\in\ell^1(G),
	\]
	so that it is continuous. Further, $\mathcal{E}$ has the tracial property
	\begin{equation}\label{e1858}
		\mathcal{E} \left(f\ast g \right) = \mathcal{E} \left(g\ast f\right)\quad\forall f,g\in\ell^1(G).
	\end{equation}
	Indeed,
	\begin{equation*}
		\begin{split}
			\mathcal{E} \left(f\ast g\right) & =
			\sum_{s\in G} f(s)g(s^{-1}e) = \sum_{s\in G}f(s)g(s^{-1})                                                                         \\
			                                 & = \sum_{s\in G}f(s^{-1})g(s) = \sum_{s\in G}g(s)f(s^{-1}e) = \mathcal{E}\left(g\ast f\right) .
		\end{split}
	\end{equation*}

	Our next concern is to prove that
	\begin{align}
		\mathcal{E}(\Phi f)                     & =\mathcal{E}(f), \label{e1828}            \\
		\mathcal{E}(\Phi f\ast\Phi f)           & =\mathcal{E}(f\ast f), \label{e1829}      \\
		\mathcal{E}(\Phi f\ast\Phi f\ast\Phi f) & =\mathcal{E}(f\ast f\ast f) \label{e1830}
	\end{align}
	for each $f\in\ell^1(G)$.
	Set $f\in \ell^1(G)$. By hypothesis, there exists a sequence $(\Phi_n)$ of
	isometric automorphisms of $\ell^1(G)$ such that
	\begin{align*}
		\Phi f                     & =\lim_{n\to\infty}\Phi_n f,              \\
		\shortintertext{and hence}
		\Phi f\ast\Phi f           & =\lim_{n\to\infty}\Phi_n(f\ast f),       \\
		\Phi f\ast\Phi f\ast\Phi f & =\lim_{n\to\infty}\Phi_n(f\ast f\ast f).
	\end{align*}
	For each $n\in\mathbb{N}$, $\Phi_n$ is expressible as
	\[
		\Phi_n g=\chi_n (g\circ\varphi_n)
		\quad
		\forall g\in\ell^1(G),
	\]
	where
	\begin{itemize}
		\item
		      $\varphi_n\colon G\to G$ is a group automorphism,
		\item
		      $\chi_n\colon G\to\mathbb{T}$ is a group homomorphism.
	\end{itemize}
	It should be pointed out that the ``Jacobian'' constant $c_n$ has been omitted in the above representation because the Haar measure on $G$ is the counting measure, so that
	$c_n=\lambda_G(\varphi_n(\{e\}))/\lambda_G(\{e\})=\lambda_G(\{e\})/\lambda_G(\{e\})=1$.
	Therefore
	\[
		\mathcal{E}(\Phi_n g)=\underbrace{\chi_n(e)}_{=1} g(\underbrace{\varphi_n(e)}_{=e})=\mathcal{E}(g)\quad\forall g\in\ell^1(G).
	\]
	We thus get
	\begin{equation*}
		\begin{split}
			\mathcal{E}(\Phi f)
			 & =\lim_{n\to\infty} \mathcal{E}(\Phi_n f)=\mathcal{E}(f),                                   \\
			\mathcal{E}(\Phi f\ast\Phi f)
			 & =\lim_{n\to\infty}\mathcal{E}\bigl(\Phi_n(f\ast f)\bigr)=\mathcal{E}(f\ast f),             \\
			\mathcal{E}(\Phi f\ast\Phi f\ast\Phi f)
			 & =\lim_{n\to\infty}\mathcal{E}\bigl(\Phi_n(f\ast f\ast f)\bigr)=\mathcal{E}(f\ast f\ast f).
		\end{split}
	\end{equation*}
	as claimed.
	For all $f,g\in \ell^1(G)$ and $\alpha\in\mathbb{C}$, we replace $f$ by $f+\alpha g$ in~\eqref{e1829}
	and~\eqref{e1830} and then we identify the coefficients of $\alpha$ in both sides of the corresponding
	equations to get
	\begin{equation*}
		\begin{split}
			\mathcal{E} \left(\Phi f\ast\Phi g+\Phi g\ast\Phi f \right) & =
			\mathcal{E} \left(f\ast g+g\ast f\right)
		\end{split}
	\end{equation*}
	and
	\begin{gather*}
		\mathcal{E} \left(\Phi f\ast\Phi f\ast\Phi g+\Phi f\ast\Phi g\ast\Phi f + \Phi g \ast \Phi f \ast \Phi f\right) \\
		=\mathcal{E}\left(f\ast f\ast g+f\ast g\ast f+g\ast f\ast f\right).
	\end{gather*}
	Using~\eqref{e1858} we arrive at
	\begin{equation}\label{e1904}
		\begin{split}
			\mathcal{E}(\Phi f\ast\Phi g)           & =
			\mathcal{E}(f\ast g),                       \\
			\mathcal{E}(\Phi f\ast\Phi f\ast\Phi g) & =
			\mathcal{E}(f\ast f\ast g).
		\end{split}
	\end{equation}

	We proceed to show that $\Phi$ is a Jordan homomorphism.
	Set $f\in\ell^1(G)$, and let $F=\Phi(f\ast f)-\Phi f\ast\Phi f$.
	Our goal is to show that $F=0$.
	For each $g\in\ell^1(G)$, we use~\eqref{e1904} to obtain
	\[
		\mathcal{E}\bigl(\Phi(f\ast f)\ast\Phi g\bigr) =
		\mathcal{E}((f\ast f)\ast g)=\mathcal{E}\bigl(\Phi f\ast\Phi f\ast\Phi g\bigr),
	\]
	so that
	\[
		\mathcal{E}\bigl(F\ast\Phi g\bigr)=0.
	\]
	Since $\Phi$ is surjective, it follows that $\mathcal{E}\bigl(F\ast h\bigr)=0$
	for each $h\in\ell^1(G)$. By defining $h\in\ell^1(G)$ by
	\[
		h(t)=\overline{F(t^{-1})}\quad\forall t\in G
	\]
	we arrive at
	\[
		0=\mathcal{E}(F\ast h)=
		\sum_{t\in G} F(t)h(t^{-1}e)=
		\sum_{t\in G}F(t)\overline{F(t)}=\sum_{t\in G}\left\vert F(t)\right\vert^2,
	\]
	which gives $F=0$, as required.
\end{proof}

\begin{theorem}\label{mainTh2}
	Let $G$ a locally compact group with $G\in {\rm [MAP]}$, and
	let $\Phi\colon L^1(G)\to L^1(G)$ be a surjective approximately local isometric automorphism.
	Then $\Phi$ is an isometric automorphism.
\end{theorem}

\begin{proof}
	On account of Theorem~\ref{mainTh3}, we are reduced to consider the case where $G$ is non-discrete.

	According to Lemma~\ref{mainlemma}, we only need to show that $\Phi$ is a Jordan homomorphism.
	To this end, let $\mathcal{A}$ be the unital closed subalgebra of $M(G)$ defined by
	\[
		\mathcal{A}=\bigl\{\alpha\delta_e+f : \alpha\in\mathbb{C}, \ f\in L^1(G)\bigr\},
	\]
	and let us define a unital surjective map $\Psi\colon\mathcal{A}\to\mathcal{A}$ by
	\[
		\Psi(\alpha\delta_e+f)=\alpha\delta_e+\Phi(f)\quad
		\forall \alpha\in\mathbb{C}, \ \forall f\in L^1(G).
	\]
	We now check that $\Psi$ preserves invertibility. Let $\alpha\delta_e+f\in\mathcal{A}$ be an invertible element,
	and take a sequence $(\Phi_{f,n})$ of isometric automorphisms of $L^1(G)$ such that $\Phi(f)=\lim\Phi_{f,n}(f)$.
	For each $n\in\mathbb{N}$, we define
	$\Psi_{f,n}\colon\mathcal{A}\to\mathcal{A}$ by
	\[
		\Psi_{f,n}(\beta\delta_e+g)=\beta\delta_e+\Phi_{f,n}(g)\quad\forall \beta\in\mathbb{C}, \ \forall g\in L^1(G).
	\]
	It is easily checked that $\Psi_{f,n}$ is an automorphism and that
	\[
		\Psi(\alpha\delta_e+ f)=\lim_{n\to\infty}\Psi_{f,n}(\alpha\delta_e+f).
	\]
	Further, we now show that $\Psi_{f,n}$ is an isometry.
	By using~\cite[Theorem~19.20(iii)]{HR}, we see that
	\begin{equation*}
		\begin{split}
			\left\Vert \Psi_{f,n}(\beta\delta_e+g) \right\Vert & =
			\left\Vert \beta\delta_e+\Phi_{f,n}(g)\right\Vert =
			\left\Vert \beta\delta_e\right\Vert + \left\Vert\Phi_{f,n}(g)\right\Vert_1                                             \\
			                                                   & = \left\Vert\beta\delta_e\right\Vert + \left\Vert g\right\Vert_1=
			\left\Vert\beta\delta_e+g\right\Vert_1
		\end{split}
	\end{equation*}
	for all $\beta\in\mathbb{C}$ and $g\in L^1(G)$.
	Since $\Psi_{f,n}$ is an automorphism, it follows that $\Psi_{f,n}(\alpha\delta_e+f)$ is invertible
	and further
	\[
		\left\Vert\Psi_{f,n}(\alpha\delta_e+f)^{-1}\right\Vert = \left\Vert\Psi_{f,n}((\alpha\delta_e+f)^{-1}) \right\Vert =
		\left\Vert (\alpha\delta_e+f)^{-1} \right\Vert \quad\forall n\in\mathbb{N}.
	\]
	From~\cite[Theorem~2.3.21(i)]{D} we deduce that $\Psi(\alpha\delta_e+f)$ is invertible.
	Since the collection of finite-dimensional irreducible representations of $\mathcal{A}$
	is a separating family for $\mathcal{A}$ (see~\cite[Theorem~3.2]{GM}),
	\cite[Th\'eor\`eme~2]{Au} shows that $\Psi$ is a Jordan homomorphism and therefore so is $\Phi$.
\end{proof}

\medskip

\noindent
\textbf{Author contributions} All authors contributed equally to the study conception and design. All authors read and approved the final manuscript.

\medskip

\noindent
\textbf{Funding} The authors were supported by
MCIN/AEI/10.13039/501100011033 and ``ERDF A way of making Europe'' grant PID2021-122126NB-C31
and by Junta de Andaluc\'{\i}a grant FQM185.

\medskip

\noindent
\textbf{Data Availability} Data sharing is not applicable to this article as no datasets were generated or analysed during the current study.

\medskip

\noindent
\textbf{Declarations}

\noindent
\textbf{Conflict of interest} The authors have no relevant financial or non-financial interests to disclose.

\end{document}